\documentclass[11pt,reqno,oneside]{amsart}
\usepackage[a4paper, total={6in, 9in}]{geometry}

\usepackage[usenames]{color}
\usepackage{amsmath,pdfsync,verbatim,graphicx,epstopdf,enumerate}
\usepackage{todonotes}
\usepackage{amsmath,amscd, amssymb, amsthm, mathrsfs}
\usepackage[abbrev,nobysame,alphabetic]{amsrefs}
\usepackage[colorlinks=true]{hyperref}
\usepackage{cancel}
\usepackage{accents}
\usepackage{tikz} \usetikzlibrary{patterns}

\hypersetup{allcolors=blue}

\newcommand{\I}{\mathrm{i}}
\DeclareMathOperator{\tr}{\mathrm{tr}}

\newcommand{\PD}{\partial}

\newcommand{\Beq}{\begin{equation}}
	\newcommand{\Eeq}{\end{equation}}
\newcommand{\beq}{\begin{equation*}}
	\newcommand{\eeq}{\end{equation*}}
\newcommand{\bal}{\begin{align}}
	\newcommand{\eal}{\end{align}}

\usepackage{mathtools}

\newtheorem{theorem}{Theorem}[section]
\newtheorem{corollary}[theorem]{Corollary}

\newtheorem{lemma}[theorem]{Lemma}
\newtheorem{proposition}[theorem]{Proposition}

\theoremstyle{definition}

\newtheorem{remark}{Remark}[section]

\newcommand{\R}{\mathbb{R}}
\newcommand{\Rn}{\mathbb{R}^n}
\newcommand{\norm}[1]{\lVert #1 \rVert}
\newcommand{\abs}[1]{\lvert #1 \rvert}

\newcommand{\supp}{\mathop{\rm supp}}
\newcommand{\dif}[1]{\,\mathrm{d}{#1}}
\newcommand{\df}{\mathrm{d}}
\newcommand{\nrm}[2][]{ \| {#2} \|_{#1}}
\newcommand{\agl}[1][\cdot]{ \langle {#1} \rangle}
\allowdisplaybreaks

\numberwithin{equation}{section}

\title[Anisotropic Calder\'on problem]{The anisotropic Calder\'on problem at large fixed frequency on  manifolds with invertible ray transform}

\author{Shiqi Ma}
\address{School of Mathematics, Jilin University, Changchun, China}
\email{mashiqi@jlu.edu.cn, mashiqi01@gmail.com}

\author{Suman Kumar Sahoo}
\address{Seminar for Applied Mathematics, Department of Mathematics, ETH Z\"urich, Switzerland}
\email{suman.sahoo@math.ethz.ch}

\author{Mikko Salo}
\address{Department of Mathematics and Statistics, University of Jyv\"askyl\"a, Jyv\"askyl\"a, Finland}
\email{mikko.j.salo@jyu.fi}

\begin{document}

\begin{abstract}
	
    We consider the inverse problem of recovering a potential from the Dirichlet to Neumann map at a large fixed frequency on certain Riemannian manifolds. We extend the earlier result of [G.~Uhlmann and Y.~Wang, arXiv:2104.03477] to the case of simple manifolds, and more generally to manifolds where the geodesic ray transform is stably invertible. The argument involves an invariantly formulated construction of Gaussian beam quasimodes with uniform bounds for the underlying constants.
	
\end{abstract}



\maketitle

\section{Introduction and statement of main results}

Let $(M,g)$ be an $n$-dimensional  $(n\ge 2)$ compact Riemannian manifold  with smooth boundary, and let $\lambda \geq 0$ be a frequency. 
We consider the boundary value problem
\begin{equation} \label{eq:1-AC22}
	\left\{ \begin{aligned}
		\mathcal{L}_{q,\lambda} u
		:= (-\Delta_g + q - \lambda^2)u&=0 \quad \mbox{in}\quad M, \\
		u&=f \quad \mbox{on} \quad \partial M,
	\end{aligned} \right.
\end{equation}
where $\Delta_g$ is the Laplace-Beltrami operator on $M$.
In local coordinates,
\[
\Delta_g u
= |g|^{-\frac{1}{2}}\partial_{j}(|g|^{\frac{1}{2}} g^{jk} \partial_{k} u),
\]
where $(g^{jk})=(g_{jk})^{-1}$ and $|g| = \det(g_{jk})$.
Suppose $\lambda^2$ is not a Dirichlet eigenvalue of $-\Delta_g +q$ in $M$, and let $ u\in C^{\infty}(M)$ be the unique solution of \eqref{eq:1-AC22} for a Dirichlet boundary condition $f \in C^{\infty}(\partial M)$.
The Dirichlet to Neumann map (DN map) associated to \eqref{eq:1-AC22} is given by
\begin{equation} \label{eq:DN-AC22}
	\Lambda_q^{\lambda}: C^{\infty}(\partial M) \to C^{\infty}(\partial M), \qquad \Lambda_{q}^{\lambda} f
	:= \partial_{\nu} u |_{\partial M}
	:= g^{jk} \nu_j \partial_{k} u |_{\partial M}. 
\end{equation}
The inverse problem we are interested in is to recover $q$ from $\Lambda_q^{\lambda}$ for a large but \emph{fixed} $\lambda$.

Before going to the statement of our main result, we first define an admissible class of perturbations for which we establish our uniqueness result.
For any nonzero $p \in H^s(M)$, we introduce a frequency function $N_s(p)$ of $p$ by
\begin{align*} 
	N_s(p)
	:= \frac {\nrm[H^s(M)]{p}} {\nrm[L^2(M)]{p}}.
\end{align*}%
If $p=0$ we define $N_s(p) = 0$. For any number $B > 0$ and any $s>0$, 
we define a set $\mathcal{A}_s(B)$ of admissible perturbations by 
\begin{equation} \label{eq:NAs-AC22}
	\mathcal A_s(B)
	:= \{p\in H^s(M) \,:\, N_s(p) \le B \}.
\end{equation}
Note that $p\in \mathcal{A}_s(B)$ implies $p \in C^1(M)$ or $p\in C^0(M)$ if $s > \frac{n}{2}+1$ or $ s > \frac{n}{2}$, respectively, by Sobolev embedding \cite{Taylor_book}*{Proposition 4.3}.

We establish two main results in this work.
Our first result is for simple manifolds (see e.g.\ \cite{PSU_book}). A compact Riemannian manifold $(M,g) $ with boundary is said to be \emph{simple} if 
(i) $(M,g)$ is nontrapping (every geodesic reaches the boundary in finite time),
(ii) $\partial M$ is strictly convex (the fundamental form of $\partial M$ is positive definite), and 
(iii) $(M,g)$ has no conjugate points.
Examples include strictly convex simply connected domains in nonpositively curved manifolds. 
Our first main result is as follows.

\begin{theorem} \label{thm:1-AC22}
	Let $(M,g)$ be a simple manifold of dimension $n \ge 2$.
	Let $B > 0$ and assume $\norm{q}_{H^s(M)} \le B$, $\norm{p}_{H^s(M)} \le B$, and $p \in \mathcal A_s(B)$, where $s > \frac{n}{2}$. 
	There is a positive constant $\lambda_0 = \lambda_0(M,g,s,B)$ such that if $\Lambda_{q+p}^\lambda =\Lambda_{q}^\lambda$ for at least one $\lambda \geq \lambda_0$, then \mbox{$p=0$ in $M$.}
\end{theorem}

\begin{remark}
	The assumption that $p \in \mathcal A_s(B)$ is similar to the assumption that the perturbation is angularly controlled in \cite{Rakesh_Uhlmann_14}*{Theorem 2} or horizontally controlled in \cite{RakeshSalo2020}. This assumption is always satisfied for some $B$ if $p$ lies in a finite dimensional space, since the norms $\nrm[L^2(M)]{p}$ and $\nrm[H^s(M)]{p}$ are equivalent in finite dimensional spaces. See the example after \cite{Rakesh_Uhlmann_14}*{Theorem 2} for infinite dimensional expansions satisfying such a condition.
	
	Moreover, the assumption $p \in \mathcal A_s(B)$ for some $s > \frac{n}{2}$  is not optimal. It might be possible to modify the argument so that $p \in \mathcal A_2(B)$ or even $p \in \mathcal A_{1+\varepsilon}(B)$ is sufficient. However, some bound on the frequency of the perturbation is needed in order to have a uniform estimate for $\lambda_0$. This  places a restriction on the perturbations that can be treated with this method.
\end{remark}

\smallskip

We note that Theorem \ref{thm:1-AC22} can also be reformulated as follows.

\begin{corollary} \label{thm:1-AC22_second}
	Let $(M,g)$ be a simple manifold of dimension $n \ge 2$, and let $q_1, q_2 \in H^s(M)$ where $s > \frac{n}{2}$.
	Then there exists a positive constant $\lambda_0$ depending on $M$, $g$, $s$, $\nrm[H^s(M)]{q_j}$, and $N_s(q_1-q_2)$ such that if $\Lambda_{q_1}^\lambda =\Lambda_{q_2}^\lambda$ for at least one $\lambda \geq \lambda_0$, then $q_1=q_2$ in $M$.
\end{corollary}

To state the second result, let us recall the notation for the geodesic ray transform following \cites{Sharafutdinov_book,PSU_book}. Assume that $(M,g)$ is nontrapping with strictly convex boundary. For a function $f$ on $(M,g)$, its geodesic ray transform is defined by
\[
If(\gamma) := \int_{s \in \gamma} f(s) \dif s
\]
where $\gamma$ ranges over the maximal geodesics on $(M,g)$.
The geodesic ray transform $I$ on $(M,g)$ is called \emph{stably invertible} (in terms of the $H^1$ norm of the ray transform) when there exists a slightly larger manifold $M_1$ with $M$ embedded in $M_1^{\mathrm{int}}$ and a positive constant $C_1$ such that
\begin{equation} \label{eq:stable-AC22}
	\nrm[L^2(M_1)]{f}
	\leq C_1 \nrm[H^1(\partial_+ S M_1)]{If}
\end{equation}
holds for all $f \in H^s(M_1)$ with $\mathrm{supp}(f) \subset M$,  for some $s>\frac{n}{2}+1$. 
On simple manifolds $(M_1,g)$ of dimension $n\ge 2$ the estimate \eqref{eq:stable-AC22} may be found e.g.\ in \cite{PSU_book}*{Theorem 4.7.8}, and related estimates even with $H^{1/2}$ norm on the right are proved in \cites{AssylbekovStefanov2020, PS_stability}.
In dimensions $n\ge 3$, if $(M_1,g)$ has strictly convex boundary and is globally foliated by strictly convex hypersurfaces, an estimate similar to \eqref{eq:stable-AC22} is proved in \cite{Uhlmann_Vasy_invention}.
Finally, for strictly convex manifolds with no conjugate points and hyperbolic trapped set, estimates similar to \eqref{eq:stable-AC22} follow from \cite{Colin_JAMS}. 
We also need the following continuity result of the geodesic ray transform which holds true at least on strictly convex nontrapping manifolds \cites{Sharafutdinov_book, PSU_book},
\begin{equation} \label{eq:stable2-AC22}
	\nrm[H^2(\partial_+ S M_1)]{If}
	\leq C_2 \nrm[H^2(M_1)]{f}.
\end{equation}
We present more details on the geodesic ray transform in Section \ref{sec:gdsc-AC22}.
The constraint in Theorem \ref{thm:1-AC22} that the manifolds must be simple can be relaxed under \eqref{eq:stable-AC22} and \eqref{eq:stable2-AC22}. As a result, a more general theorem follows. 

\begin{theorem} \label{thm:2-AC22}
	Let $(M,g)$ be a compact nontrapping Riemannian manifold of dimension $ n \ge 2$ with smooth boundary.
	Suppose the geodesic ray transform is stably invertible and continuous, i.e.~\eqref{eq:stable-AC22} and \eqref{eq:stable2-AC22} are satisfied.
	Assume $\norm{q}_{H^s(M)} \le B$, $\norm{p}_{H^s(M)} \le B$, $p \in \mathcal{A}_s(B)$ with $s > 1+\frac{n}{2}$. Then there exists a positive constant $\lambda_0 = \lambda_0(M,g,s,B)$ such that if $\Lambda_{q+p}^\lambda =\Lambda_{q}^\lambda$ holds for at least one $\lambda \geq \lambda_0$, then  $p=0$ in $M$.
\end{theorem}






We now provide a brief survey of the existing results of the Calder\'on problem.
In the Euclidean setting there is a substantial literature on such problems and we refer the readers to the survey \cite{Uhl_survey}.  In this work we are interested in the anisotropic problem, which can be understood as an inverse problem for the equation $\nabla \cdot (\gamma \nabla u)=0$ where $\gamma$ is a positive definite matrix function, or as an inverse problem for the Laplace-Beltrami equation or for the Schr\"odinger equation $(-\Delta_g + q)u = 0$ on a Riemannian manifold. If the manifold and coefficients are real-analytic, they can recovered from the DN map \cites{LeeUhlmann1989, LassasUhlmann2001, LassasTaylorUhlmann2003}. In the smooth case it is known for $n = 2$ that a potential $q$ can be determined from the DN map $\Lambda_q^\lambda$ for a fixed frequency $\lambda \geq 0$ \cite{GuillarmouTzou2011}. For $n \geq 3$ this is an open problem, however there are partial results 
in the class of admissible manifolds as well as conformally transversally anisotropic (CTA) manifolds.

We say that $ (M,g)$ is a CTA manifold if $(M,cg) \subset (\mathbb{R} \times M_0, e \oplus g_0)$, where $c$ is a smooth positive scalar function,  $e$ is the Euclidean metric, and $(M_0,g_0)$ is an $(n-1)$-dimensional manifold.
We say $(M,g)$ is admissible if additionally the transversal manifold $(M_0,g_0)$ is simple.
Theorem \ref{thm:1-AC22} for any $\lambda \ge 0$ has been proved on admissible manifolds in \cite{DKSU}, whereas \cite{Dos_Jems} proved the corresponding uniqueness result on CTA manifolds. These methods are based on a geometric version of complex geometrical optics solutions introduced in \cite{SYL}  in the Euclidean case. Related recent results are given in \cites{CFO, FKO}.



In our setting the manifolds do not satisfy the additional product structure mentioned above, and thus complex geometrical optics solutions are not available. However, when the frequency $\lambda > 0$ is very large there exist traditional geometrical optics type solutions, and one can construct such solutions that concentrate along geodesics. If one could take the limit $\lambda \to \infty$ then one could recover the geodesic ray transform of the perturbation $p$. In our case the frequency $\lambda$ is large but fixed, and we will instead use the condition $p \in \mathcal A_s(B)$ to recover the ray transform. These ideas were used in \cite{Uhlmann_Wang_2022} combined with an analysis of the semiclassical resolvent in order to prove a similar result for nonpositively curved manifolds when $n=3$. We give a direct argument based on geometrical optics and Gaussian beam constructions, and obtain results on any manifold with stably invertible geodesic ray transform in any dimension.

In Section \ref{sec:UC-AC22} we present a Gaussian beam construction with uniform bounds for the underlying constants. This is a key component for proving Theorem \ref{thm:2-AC22}. To achieve this, we express the Riccati and transport ODEs for the phase and amplitude functions of the Gaussian beam in an invariant manner. This ensures that the bounds will not depend on choices of (Fermi) coordinates. Finally, by utilizing energy estimates we are able to obtain the desired uniform bounds.

The rest of the article is structured as follows.
Section \ref{sec:gdsc-AC22} contains some preliminary results related to the geodesic ray transform.
In Section \ref{sec:resol-AC22} we present the proof of a resolvent estimate on non-trapping manifolds. Section \ref{sec:smpl-AC22} gives a construction of special solutions of \eqref{eq:1-AC22} on simple manifolds and proves Theorem \ref{thm:1-AC22}.
The proof of Theorem \ref{thm:2-AC22} is contained in Section \ref{sec:Thm2-AC22}.
In Remarks \ref{rem:sta1-AC22} and \ref{rem:sta2-AC22} we discuss the stability of the inverse problem.
Section \ref{sec:UC-AC22} gives the invariant construction of Gaussian beams with uniform bounds required for Theorem \ref{thm:2-AC22}.

\subsection*{Acknowledgements}

The authors would like to express their deep gratitude to Katya Krupchyk and Simon St-Amant for several helpful discussions, in particular related to uniform bounds for Gaussian beams. The last two authors would also like to thank the Isaac Newton Institute for support and hospitality during the programme Rich and nonlinear tomography (EPSRC grant EP/R014604/1) when part of this work was undertaken. All the authors were partly supported by the Academy of Finland (Centre of Excellence in Inverse Modelling and Imaging, grant 284715) and by the European Research Council under Horizon 2020 (ERC CoG 770924). 

\section{preliminaries on geodesic ray transform} \label{sec:gdsc-AC22}


In this section, we recall the geodesic ray transform and several facts related to it.
We refer readers to \cites{Sharafutdinov_book, PSU_book} for more information on the geodesic ray transform.

Let $M$ be a compact manifold with smooth boundary and let $T_x M$ be the tangent space attached to the point $x \in M$. We write the $g$-inner product for tangent or cotangent vectors as $\agl[\,\cdot\,, \,\cdot\,] = \agl[\,\cdot\,, \,\cdot\,]_g$.
We also write $|\,\cdot\,| = |\,\cdot\,|_g := \agl[\,\cdot\,, \,\cdot\,]_g^{1/2}$.
Sobolev spaces such as $H^k(M)$, $H_0^k(M)$ and $L^2(M)$ can be defined in a similar manner as in the Euclidean setting, and readers may refer to \cite{Taylor_book}*{Chapter 4} for more details. 

The unit sphere bundle $SM$ of $M$ is defined as 
\begin{equation*}
	SM := \bigcup_{x\in M} S_{x} M \quad \mbox{where} \quad S_{x}M :=\{(x,v)\in T_{x} M \,; \,|v|=1\}.
\end{equation*}
If the dimension of $M$ is $n$ then the dimension of $SM$ will be $2n-1$.
The boundary of $SM$, denoted as $\partial (SM)$, is defined as
\(
\partial(SM)
= \{(x,v)\in SM \,;\, x\in \partial M \},
\)
and it is the union of the sets of inward and outward pointing vectors,
\[
\partial_{\pm} S M = \{(x,v)\in \partial(SM) \,;\, \pm\langle v,\nu\rangle \le 0\}.
\]
Here $\nu$ is the outward unit normal to the boundary $ \partial M$. We equip $SM$ with the Sasaki metric induced by $g$, and this yields natural volume forms $\dif(SM)$ and $\dif(\partial SM)$.

A unit speed geodesic starting at $x$ and moving in the direction $v$ is denoted by $t \mapsto \gamma(t,x,v)$.
Let $\tau(x,v)$ be the time when $\gamma$ exits $M$.
We say that $(M,g)$ is nontrapping if $ \tau(x,v)$ is finite for all $ (x,v)\in SM$, and that $(M,g)$ is strictly convex if the second fundamental form on $\partial M$ is positive definite. We also define the geodesic flow $\varphi_t$ on $SM$ by $\varphi_{t}(x,v) := (\gamma(t,x,v),\dot{\gamma}(t,x,v))$.

Let $(M,g)$ be strictly convex and nontrapping. The geodesic ray transform $I: C^{\infty} (M)\rightarrow C^{\infty} (\partial_{+}SM)$ is a linear map given by
\begin{equation*}
    If(x,v) := \int_{0}^{\tau(x,v)} f(\gamma(t,x,v)) \dif t.
\end{equation*}
We recall the Santal\'o formula and the expression of the adjoint of $I$.

\begin{lemma} \label{lem:San1-AC22}
	Let $F \colon SM \rightarrow \mathbb{R}$ be a continuous function. Then we have
	\begin{align}
		\int_{SM} F \dif{(SM)}
		= \int_{\partial_{+}SM}\int_{0}^{\tau(x,v)} F(\varphi_{t}(x,v)) \, \mu (x,v)  \dif t \dif{(\partial SM)},
	\end{align}
	where $ \mu(x,v):=-\langle v, \nu(x) \rangle $.
\end{lemma}

\begin{lemma} \label{lem:San2-AC22}
    Let $ f\in C^{\infty}(M)$ and $ h\in C_0^{\infty}((\partial_{+} SM)^{int})$. Then 
    \begin{align}
        (If, h)_{L_{\mu}^2 (\partial_{+}SM)} = (f, I^*h).
    \end{align}
    Here $I^*h$ is given by 
    \(
    \displaystyle I^*h(x) = \int_{S_x} h_{\psi} (x,v) \dif v
    \)
    where $h_{\psi}(x,v)= h(\varphi_{-\tau(x,-v)}(x,v))$ for all $(x,v)\in SM$, and $L^2_{\mu}(\partial_+ SM)$ is the $L^2$ space with measure $\mu \dif(\partial SM)$.
\end{lemma}

\begin{lemma} \label{lemma_i_bounded}
	For every non-negative integer $k$, the ray transform $I$ is a bounded linear operator from $H^k(M)$ to $H^k(\PD_{+}SM)$.
\end{lemma}

The proofs of these results can be found in \cite{PSU_book}*{Chapters 3 and 4}. We will also need the following facts on the normal operator of the geodesic ray transform on simple manifolds which follow from \cite{Stefanov_Uhlmann_duke} (see also \cite{PSU_book}*{Chapter 8}).

\begin{lemma} \label{lemma_simple_normal}
Let $(M,g)$ be a simple manifold. Then $I^* I$ is an elliptic pseudodifferential operator of order $-1$ in $M^{\mathrm{int}}$. Given $s \in \R$ and a compact set $K \subset M^{\mathrm{int}}$, there is $C > 0$ so that one has the inequalities 
\[
C^{-1} \norm{f}_{H^s(M)} \leq \norm{I^* I f}_{H^{s+1}(M)} \leq C \norm{f}_{H^s(M)}
\]
for any $f \in H^s(M)$ with $\mathrm{supp}(f) \subset K$.
\end{lemma}


\section{Resolvent estimate} \label{sec:resol-AC22}

The proofs of the main theorems are based on constructing approximate geometrical optics or Gaussian beam type solutions. In order to convert these approximate solutions to exact solutions, we will need the following solvability result at high frequencies.

\begin{proposition} \label{prop:resol-AC22}
    Let $(M,g)$ be a compact nontrapping manifold with smooth boundary, and let 
    $q \in L^\infty(M)$ with $\nrm[L^\infty(M)]{q} \leq B$.
    There are $C = C(M,g) > 0$ and $\lambda_0 = \lambda_0(M,g,B) > 0$ so that for any $\lambda \geq \lambda_0$ and any $f \in L^2(M)$, the equation 
    \[
    (-\Delta_g   - \lambda^2+q) u = f \text{ in $M$}
    \]
    has a solution $u \in H^2(M)$ with
    \[
    \lambda \norm{u}_{L^2(M)} + \norm{\df u}_{L^2(M)} + \lambda^{-1} \norm{\nabla^2 u}_{L^2(M)}
    \leq C\nrm[L^2(M)]{f}.
    \]
\end{proposition}

The estimate given in Proposition \ref{prop:resol-AC22} resembles a resolvent estimate in scattering theory, where it is well known that a nontrapping assumption is required for such an estimate to hold.
These estimates are typically given on noncompact manifolds with suitable assumptions at infinity. See e.g.~\cite{Wunsch2012} for a discussion on such estimates (note that if one excludes a small set of frequencies, this kind of estimate may hold for general geometries \cite{LSW21}).
Our estimate on compact manifolds with boundary is even simpler, and we give a proof based on a positive commutator argument.
For the proof it is convenient to switch to semiclassical notation and write $h = \lambda^{-1}$. See \cite{zworski2012semi} for the semiclassical analysis facts used below.

We may assume that $M$ is embedded in a closed manifold $(N,g)$ having the same dimension, and for all $s \in \R$ we may consider the semiclassical Sobolev norm
\[
\norm{u}_{H^s_{\mathrm{scl}}(N)} = \norm{(I-h^2 \Delta_g)^{s/2} u}_{L^2(N)}
\]
where $(I-h^2 \Delta_g)^{s/2}$ is defined via the spectral theorem.
Proposition \ref{prop:resol-AC22} will follow by a standard duality argument from the next {\it a priori} estimate with $s=0$ (see e.g.~\cite{DKSU}*{Proposition 4.4} for this duality argument). We employ a generic constant $C$ throughout the manuscript, the value of which may vary from line to line.


\begin{lemma} \label{lem:resol-AC22}
	Let $(M,g)$ be a compact nontrapping manifold with smooth boundary, and let $M$ be embedded in a closed manifold $(N,g)$ having the same dimension.
	Let $s \in \R$.
	There are $C > 0$, $h_0 > 0$ such that for $0 < h \leq h_0$, one has 
	\[
	h \norm{u}_{H^{s+2}_{\mathrm{scl}}(N)} \leq C \norm{(-h^2 \Delta_g - I) u}_{H^s_{\mathrm{scl}}(N)}, \quad  u \in C^{\infty}_c(M^{\mathrm{int}}).
	\]
\end{lemma}

\begin{proof}
	We first prove the estimate for $s=0$.
	Write $P = -h^2 \Delta_g$ and decompose $u$ as 
	\[
	u = Bu + (I-B)u
	\]
	where $B$ is a semiclassical pseudodifferential operator obtained by quantizing the symbol $b(x,\xi) := \psi(|\xi|_g) \in C^{\infty}(T^* N)$ where $\psi \in C^{\infty}_c(\R)$ with $\psi(t) = 1$ near $t = 1$ and $\psi = 0$ outside a small neighborhood of $t=1$.
	Denote the semiclassical principal symbol of $P-I$ by $r(x,\xi)$, so $r(x,\xi) = |\xi|_g^2 - 1$.
	Since $P-I$ is semiclassically elliptic away from $\{ |\xi|_g = 1 \}$, we can find a symbol $q(x,\xi)$ of order $-2$ such that $q = r^{-1}$ in $\supp (1-b)$.
	This implies
	\[
	1 - b(x,\xi)
	= (1-b(x,\xi)) q(x,\xi) r(x,\xi), \quad (x,\xi) \in TM.
	\]
	By semiclassical calculus, see \cite{zworski2012semi}*{\S 14.2}, we have 
	\[
	I-B = \mathrm{Op}((1-b) q) \mathrm{Op}(r) + h \Psi^{-1} = \mathrm{Op}((1-b) q) (P-I) + h \Psi^{-1}.
	\]
	From this one obtains the estimate 
	\begin{align*}
		\norm{(I-B)u}_{H^2_{\mathrm{scl}}(N)} &\leq \norm{\mathrm{Op}((1-b) q) (P-I) u}_{H^2_{\mathrm{scl}}(N)} + Ch \norm{u}_{H^1_{\mathrm{scl}}(N)}  \\
		&\leq C \norm{(P-I) u}_{L^2(N)} + C h \norm{u}_{H^1_{\mathrm{scl}}(N)} 
	\end{align*}
	valid for $u \in C^{\infty}(N)$. Writing $u = (I-B)u + Bu$ on the right, it follows that 
	\begin{equation}
	\norm{(I-B)u}_{H^2_{\mathrm{scl}}(N)} \leq C \norm{(P-I) u}_{L^2(N)} + Ch \norm{Bu}_{H^1_{\mathrm{scl}}(N)} \label{eq:imbu-AC22}.
	\end{equation}
	
	We now proceed to an estimate for $Bu$, which is microlocalized to a small neighborhood of $\{ |\xi|_g = 1 \}$. To do this we invoke the positive commutator method. Assume that we can find a formally self-adjoint linear operator $A: C^{\infty}(N) \to C^{\infty}(N)$ such that 
	\begin{align*}
		\norm{Au}_{L^2(N)} &\leq C \norm{u}_{H^1_{\mathrm{scl}}(N)}, \\
		(\I [P,A] u, u)_{L^2(N)} &\geq c h \norm{Bu}_{H^1_{\mathrm{scl}}(N)}^2 - C h \norm{(I-B)u}_{H^1_{\mathrm{scl}}(N)}^2,
	\end{align*}
	for any $u \in C^{\infty}_c(M^{\mathrm{int}})$ and $0 < h \leq h_0$. We can then make the following computation:
	\begin{align*}
		c h \norm{Bu}_{H^1_{\mathrm{scl}}(N)}^2 - C h \norm{(I-B)u}_{H^1_{\mathrm{scl}}(N)}^2 &\leq (\I [P,A] u, u) = (\I [P-I,A] u, u) \\
		&= \I (Au, (P-I)u) - \I ((P-I)u, Au).
	\end{align*}
	By using Cauchy-Schwarz with $\epsilon$, and since $\norm{Au}_{L^2(N)} \leq C \norm{u}_{H^1_{\mathrm{scl}}(N)}$, we have
	\begin{equation} \label{eq:BAu-AC22}
		c h \norm{B u}_{H^1_{\mathrm{scl}}(N)}^2
		\leq \epsilon h \norm{u}_{H^1_{\mathrm{scl}}(N)}^2 + \frac{1}{\epsilon h} \norm{(P-I)u}_{L^2(N)}^2 + C h \norm{(I-B)u}_{H^1_{\mathrm{scl}}(N)}^2.
	\end{equation}
	Therefore,
	\begin{align*}
		&h^2 \norm{u}_{H^1_{\mathrm{scl}}(N)}^2
		  \leq 2 h^2 \norm{Bu}_{H^1_{\mathrm{scl}}(N)}^2 + 2h^2 \norm{(I-B)u}_{H^1_{\mathrm{scl}}(N)}^2 \\
		& \leq C\epsilon h^2 \norm{u}_{H^1_{\mathrm{scl}}(N)}^2 + C \epsilon^{-1} \norm{(P-I)u}_{L^2(N)}^2 + C h^2 \norm{(I-B)u}_{H^1_{\mathrm{scl}}(N)}^2 \quad (\text{by~} \eqref{eq:BAu-AC22}) \\
		& \leq h^2(C\epsilon + Ch^2) \norm{u}_{H^1_{\mathrm{scl}}}^2 + (C\epsilon^{-1} + C h^2) \norm{(P-I) u}_{L^2(N)}^2. \quad (\text{by~} \eqref{eq:imbu-AC22})
	\end{align*}
	Choosing the value of $\epsilon$ so that $C \epsilon = 1/2$, we obtain the estimate 
	\begin{align*}
		h^2 \norm{u}_{H^1_{\mathrm{scl}}(N)}^2
		& \leq C \norm{(P-I)u}_{L^2(N)}^2.
	\end{align*}
	valid for all $u \in C^{\infty}_c(M^{\mathrm{int}})$ as long as one can find an operator $A$ satisfying the conditions given above.
	
	We construct the conjugate operator $A$ as a first order semiclassical pseudodifferential operator, 
	obtained as the Weyl quantization of a real valued symbol $a \in C^{\infty}(T^* N)$. The semiclassical principal symbol of $\I h^{-1} [P,A]$ is $\{p, a\} = H_p a$, where $p = |\xi|_g^2$ is the principal symbol of $P$ and $H_p$ is the Hamilton vector field of $p$.
	The assumption that $(M,g)$ is nontrapping means precisely that there is a function $a \in C^{\infty}(S^*M)$ (escape function) with $H_p a > 0$ in $S^* M$, where $S^* M$ denotes the unit cosphere bundle.
	See e.g.~\cite{DH72}*{Theorem 6.4.1}.
	We extend $a$ smoothly to $T^* N$ as a symbol that is homogeneous of degree one for $|\xi| \geq 1$. By continuity the function $H_p a$ satisfies
	\begin{equation} \label{eq:Hpa-AC22}
		H_p a(x,\xi) \geq c |\xi|_g^2, \quad \xi \in T^* M_1, \ |\xi|_g \sim 1,
	\end{equation}
	for some compact set $M_1 \subset N$ with $M \subset M_1^{\mathrm{int}}$.
	Note that \eqref{eq:Hpa-AC22} holds only for $|\xi|_g$ away from $0$, and we shall only apply this to $Bu$ which is supported near $|\xi|_g = 1$ in the phase space.
	Quantizing $a$ gives a semiclassical operator $A$ of order one.
	Using the semiclassical G{\aa}rding inequality \cite{zworski2012semi}*{Theorem 4.30} for $(\I h^{-1} [P,A] Bu, Bu)$ and Cauchy-Schwarz with $\epsilon$ for the other terms gives that
	\begin{align}
		(\I h^{-1} [P,A]u, u)
		& = (\I h^{-1} [P,A]Bu, Bu) + (\I h^{-1} [P,A]Bu, (I-B)u) \nonumber \\
		& \quad + (\I h^{-1} [P,A](I-B)u, Bu) + (\I h^{-1} [P,A](I-B)u, (I-B)u) \nonumber \\
		& \geq c \norm{Bu}_{H^1_{\mathrm{scl}}}^2 - \nrm[H_{\mathrm{scl}}^{-1}(N)]{\I h^{-1} [P,A] Bu} \nrm[H_{\mathrm{scl}}^1(N)]{(I-B)u} \nonumber \\
		& \quad - \nrm[H_{\mathrm{scl}}^{-1}(N)]{\I h^{-1} [P,A] (I-B)u} \big( \nrm[H_{\mathrm{scl}}^1(N)]{Bu} + \nrm[H_{\mathrm{scl}}^1(N)]{(I-B)u} \big) \nonumber \\
		& \geq c \norm{Bu}_{H^1_{\mathrm{scl}}}^2 - C \norm{(I-B)u}_{H^1_{\mathrm{scl}}}^2, 
		\label{eq:iPAu-AC22}
	\end{align}
	for all $u \in C^{\infty}_c(M^{\mathrm{int}})$.
	Here we used that $\I h^{-1} [P,A]$ is of order $2$ so it is a bounded map from $H^1_{\mathrm{scl}}(N)$ to $H^{-1}_{\mathrm{scl}}(N)$.
	This completes the construction of $A$.
	We have so far proved the following estimate for all $u \in C^{\infty}_c(M^{\mathrm{int}})$:
	\[
	h \norm{u}_{H^1_{\mathrm{scl}}(N)}
		\leq C \norm{(P-I)u}_{L^2(N)}.
	\]
	
	To prove the analogous estimate for general $s$, we may apply the above estimate in a small extension $(M_1,g)$ of $(M,g)$ (which is still nontrapping) to the function $\chi(I-h^2 \Delta_g)^{s/2} u$ where $\chi \in C^{\infty}_c(M_1^{\mathrm{int}})$ satisfies $\chi=1$ near $M$, and $u \in C^{\infty}_c(M^{\mathrm{int}})$. 
	Commuting the cutoff $\chi$ to the other side of $P-I$ produces commutator terms that are $O(h^{\infty})$ by pseudolocality and the support properties of $u$ and $\df \chi$, and these can be absorbed. See e.g.~\cite{DKSU}*{Lemma 4.3} for details.
	This argument gives
	\begin{equation} \label{eq:reso2-AC22}
		h \norm{u}_{H^{s+1}_{\mathrm{scl}}(N)} \leq C \norm{(-h^2 \Delta_g - I) u}_{H^s_{\mathrm{scl}}(N)}, \quad  u \in C^{\infty}_c(M^{\mathrm{int}}).
	\end{equation}
	
	Finally, to improve the left hand side of \eqref{eq:reso2-AC22} from $s+1$ to $s+2$, we do the following computation:
	\begin{align*}
		h\norm{u}_{H^{s+2}_{\mathrm{scl}}(N)}
		& = h \norm{(-h^2 \Delta_g - I + 2I) u}_{H^s_{\mathrm{scl}}(N)}
		\leq h \norm{(-h^2 \Delta_g - I) u}_{H^s_{\mathrm{scl}}(N)} + 2h \norm{u}_{H^s_{\mathrm{scl}}(N)} \nonumber \\
		& \leq h \norm{(-h^2 \Delta_g - I) u}_{H^s_{\mathrm{scl}}(N)} + C \norm{(-h^2 \Delta_g - I) u}_{H^{s-1}_{\mathrm{scl}}(N)} \nonumber \\
		& \leq C \norm{(-h^2 \Delta_g - I) u}_{H^s_{\mathrm{scl}}(N)}, \quad u \in C^{\infty}_c(M^{\mathrm{int}}),
	\end{align*}
	where in the second last line we used \eqref{eq:reso2-AC22}. The proof is complete.
\end{proof}

\begin{corollary} \label{cor:resol-AC22}
	Assume the conditions in Lemma \ref{lem:resol-AC22}, let $-2 \leq s \leq 0$, and let 
	$q \in L^\infty(M)$ with $\nrm[L^{\infty}(M)]{q} \le B$.
	Then there are $C = C(M,g,s) > 0$ and $h_0 = h_0(M,g,s,B) > 0$ such that for $0 < h \leq h_0$ one has 
	\[
	h \norm{u}_{H^{s+2}_{\mathrm{scl}}(N)}
	\leq C \norm{(-h^2 (\Delta_g - q(x)) - I) u}_{H^s_{\mathrm{scl}}(N)}, \quad u \in C^{\infty}_c(M^{\mathrm{int}}).
	\]
\end{corollary}

\begin{proof}
	We have
	\(
	\nrm[H^s_{\mathrm{scl}}(N)]{q u}
	\leq \nrm[L^2(N)]{q u}
	\leq \nrm[L^\infty(N)]{q} \nrm[H^{s+2}_{\mathrm{scl}}(N)]{u}
	\)
	provided $-2 \leq s \leq 0$.
	Then by Lemma \ref{lem:resol-AC22} we have
	\begin{align*}
		\norm{(-h^2 (\Delta_g - q(x)) - I) u}_{H^s_{\mathrm{scl}}(N)}
		& \geq \norm{(-h^2 \Delta_g - I) u}_{H^s_{\mathrm{scl}}(N)} - h^2 \norm{q u}_{H^s_{\mathrm{scl}}(N)} \\
		& \geq ch \norm{u}_{H^{s+2}_{\mathrm{scl}}(N)} - \nrm[L^\infty(N)]{q} h^2 \norm{u}_{H^{s+2}_{\mathrm{scl}}(N)}.
	\end{align*}
	Choosing $h_0 = c/(2 B)$ completes the proof.
\end{proof}


Now we are ready to prove Proposition \ref{prop:resol-AC22}.

\begin{proof}[Proof of Proposition \ref{prop:resol-AC22}]
	Denote $\mathbf{E} = \mathcal{L}_{q,h^{-1}} (C^{\infty}_c(M^{\mathrm{int}}))$. Then $\mathbf{E}$ is a subspace of $H^{-2}_{\mathrm{scl}}(N)$, and for $h$ small any element of $\mathbf{E}$ can be written uniquely as $\mathcal{L}_{q,h^{-1}} u$ for some $u \in C^{\infty}_c(M^{\mathrm{int}})$ by Corollary \ref{cor:resol-AC22}.
	Let $f  \in L^2(M)$, and define the linear operator $T \colon \mathbf{E} \rightarrow \mathbb R$ by
	\begin{align*}
		T(\mathcal{L}_{q,h^{-1}}^* z)
		= \langle f,z \rangle_{L^2(M)}, \quad  z \in C^{\infty}_c(M^{\mathrm{int}}),
	\end{align*}
	where $\mathcal{L}_{q,h^{-1}}^*$ is the dual operator of $\mathcal{L}_{q,h^{-1}}$.
	We have $\mathcal{L}_{q,h^{-1}}^* = \mathcal{L}_{q,h^{-1}}$.
	Corollary \ref{cor:resol-AC22} gives
	\begin{align*}
		|T(\mathcal{L}_{q,h^{-1}}^* z)|
		\leq \nrm[L^2(M)]{f} \nrm[L^2(M)]{z}
		\leq \nrm[L^2(M)]{f} C h \nrm[H^{-2}_{\mathrm{scl}}(N)]{\mathcal{L}_{q,h^{-1}}^* z}.
	\end{align*}
	This implies $T$ is a bounded linear operator on $\mathbf{E}$, thus by the Hahn-Banach theorem there exists a linear functional $ \hat{T}$ on $H^{-2}_{\mathrm{scl}}(N)$ that extends $T$ from $\mathbf{E} $ to $H^{-2}_{\mathrm{scl}}(N)$ such that 
	\begin{align*}
		\nrm{\hat T}
		\leq C h \nrm[L^2(M)]{f}.
	\end{align*}
	Because $H^{-2}_{\mathrm{scl}}(N)$ is the dual space of $H^{2}_{\mathrm{scl}}(N)$ and it is a Hilbert space, by the Riesz representation theorem there exists a function $v \in H^{2}_{\mathrm{scl}}(N)$ such that $\hat{T}(z)= \langle v,z \rangle$ for all $z \in C^{\infty}_c(M^{\mathrm{int}})$.
	Furthermore, $\nrm[H^{2}_{\mathrm{scl}}(N)]{v} = \nrm{\hat T} \leq C h \nrm[L^2(M)]{f}$.
	Now set $v|_{M}= u$, then for all $z \in C^{\infty}_c(M^{\mathrm{int}})$ we have
	\[
	\agl[\mathcal{L}_{q,h^{-1}} u, z]
	= \agl[u, \mathcal{L}_{q,h^{-1}}^* z]
	= \agl[v, \mathcal{L}_{q,h^{-1}}^* z]
	= \hat T(\mathcal{L}_{q,h^{-1}}^* z)
	= T(\mathcal{L}_{q,h^{-1}}^* z)
	= \agl[f,z].
	\]
	This gives existence of a solution with the desired estimate. 
\end{proof}

\section{The case of simple manifolds} \label{sec:smpl-AC22}

In this section we construct special solutions of \eqref{eq:1-AC22} on a simple manifold following arguments in \cite{DKSU}, and give the proof of Theorem \ref{thm:1-AC22}.


\subsection{Special solutions on simple manifolds} \label{subsec:simp-AC22}

Let $(M,g)$ be a simple manifold. We wish to construct solution of \eqref{eq:1-AC22} in the form of $u = e^{\I \lambda \phi}a + R$.
A straightforward computation gives
\begin{equation} \label{eq:La-AC22}
	\mathcal{L}_{q,\lambda} (e^{\I \lambda \phi} a)
	= e^{\I\lambda \phi} \big[ \lambda^2(|\df \phi|_{g}^2-1) a - \lambda \mathcal T_{g,\phi} a - (\Delta_g - q)a \big], \ \
	\mathcal T_{g,\phi}
	:= 2\I \agl[\df \phi, \df \cdot]_g + \I\Delta_g \phi.
\end{equation}
Here $\mathcal T_{g,\phi}$ is a first-order linear differential operator depending on $g$ and $\phi$.
Substituting the ansatz $u = e^{\I \lambda \phi}a + R$ into  \eqref{eq:1-AC22}, with the help of \eqref{eq:La-AC22} we see that $\mathcal{L}_{q,\lambda} u = 0$ provided that 
\begin{equation} \label{eq:lam0-AC22}
	\mathcal{L}_{q,\lambda} R
	= e^{\I\lambda \phi} \big[ -\lambda^2(|\df \phi|_{g}^2-1) a + \lambda \mathcal T_{g,\phi} a + (\Delta_g - q) a \big] \ \ \mbox{in} \ \ M.
\end{equation}
We shall construct a real-valued phase function $\phi$ and an amplitude $a$ by making the coefficients of $\lambda^2$ and $\lambda$ in \eqref{eq:lam0-AC22} to be zero so that \eqref{eq:lam0-AC22} can be simplified.

First, we solve $|\df \phi|_{g}^2 = 1$.
This non-linear PDE is known as the eikonal equation.
Since $M$ is simple, we can extend $M$ to a larger simple compact manifold $M_1$ such that $ M \subset M^{\mathrm{int}}_1$, where $M^{\mathrm{int}}_1$ signifies the interior of $M_1$.
Let $ y\in \partial M_1$. By the properties of simple manifolds \cite{PSU_book}*{Section 3.8}, any $x$ belonging to $M_1$ can be expressed as $x= \exp_{y}(r\theta)$ with certain $r>0$ and $\theta \in S_y M := \{\xi \in T_{y}M \,;\, |\xi|_{g}=1\}$.
Here the map $\exp_y$, parameterized by $y$, is the exponential map defined on $M_1$,
and it defines the so-called polar normal coordinates on $M$ by identifying $x$ with the coordinates $(r,\theta) \in \R^+ \times S_y M$.
In these coordinates, the metric $g$ can be represented as
\[
g|_{(r,\theta)} = \df r^2+ g_0(r,\theta) \df \theta^2
\]
where $g_0$ is another positive-definite Riemannian metric, and there holds $|g| = |g_0|$.
The coordinate $r$ can be used to define a distance function from a point $x$ to $y$ by setting $\mathrm{dist}_{g}(x,y) := r$.
We now choose
\begin{equation} \label{eq:dist-AC22}
	\phi(x)
	= \pm \mathrm{dist}_{g}(x,y)
	= \pm r, \quad x \in M, \, y \in \partial M_1,
\end{equation}
thus $\partial_\theta \phi = 0$, and so $|\df \phi|_{g}^2 = (\pm \partial_r r)^2 = 1$.
Hence the eikonal equation is solved, and we can simplify \eqref{eq:lam0-AC22} to 
\begin{equation} \label{eq:lam12-AC22}
	\mathcal{L}_{q,\lambda} R
	= e^{\I\lambda \phi} [\lambda \mathcal T_{g,\phi} a + (\Delta_g - q) a].
\end{equation}

Second, we fix an integer $J \in \mathbb N$, set $a_{-1} \equiv 0$ and look for an amplitude $a$ having the form $a = \sum_{j=-1}^J \lambda^{-j} a_j$. After substituting this into \eqref{eq:lam12-AC22}, it follows that
\begin{equation} \label{eq:lam1-AC22}
	\mathcal{L}_{q,\lambda} R
	= e^{\I\lambda \phi} \sum_{j=0}^{J} \lambda^{-j+1} [\mathcal T_{g,\phi} a_j + (\Delta_g - q) a_{j-1}] + e^{\I\lambda \phi} \lambda^{-J} (\Delta_g - q) a_J, \ \ \mbox{in} \ \ M.
\end{equation}
Because $a_{-1} \equiv 0$, the following transport equations for $a_j$ can be solved iteratively starting from $j = 0$ until $j=J$:
\begin{equation} \label{eq:trans-AC22}
	\mathcal T_{g,\phi} a_j = (-\Delta_g + q) a_{j-1}.
\end{equation}
Recall \eqref{eq:dist-AC22} and $\mathcal T_{g,\phi}$ defined in \eqref{eq:La-AC22}.
By the choice of $\phi$ we have $\agl[\df \phi, \df a_j]_g = \pm \partial_r a_j$.
This reduces the equation \eqref{eq:trans-AC22} to
\[
\pm 2 \I \partial_r a_j \pm \I |g|^{-\frac 1 2} \partial_r (|g|^{\frac 1 2}) a_j = (-\Delta_g + q) a_{j-1}
\ \Leftrightarrow \
\partial_r (|g|^{\frac 1 4} a_j) = \mp \I |g|^{\frac 1 4} (-\Delta_g + q) a_{j-1} / 2,
\]
which implies for $j=0,1,\cdots,J$,
\begin{equation} \label{eq:aj-AC22}
	a_j(r,\theta)
	= |g(r,\theta)|^{-1/4} \big[ b_j(\theta) \mp \frac {\I} 2 \int_0^r |g(s,\theta)|^{1/4} (-\Delta_g + q(s,\theta)) a_{j-1}(s,\theta) \dif s \big],
\end{equation}
where $b_j$ are any smooth functions.
Especially, due to $a_{-1} \equiv 0$ we have
\begin{equation} \label{eq:a0-AC22}
	a_0(r,\theta)
	= |g(r,\theta)|^{-1/4} b(\theta),
\end{equation}
where $b$ is a smooth function.
Readers may note that $a_0$ is independent of the potential.
After solving $a_j$, we can substitute \eqref{eq:trans-AC22} into \eqref{eq:lam1-AC22} to further reduce the original equation to
\begin{equation*} \label{eq:lam2-AC22}
	\mathcal{L}_{q,\lambda} R
	= e^{\I\lambda \phi} \lambda^{-J} (\Delta_g - q) a_J \ \ \mbox{in} \ \ M,
\end{equation*}
where $a_J$ is determined by \eqref{eq:aj-AC22}.
By Proposition \ref{prop:resol-AC22}, for $\lambda$ large there is $R \in H^2(M)$ solving the above equation such that $\nrm[L^2(M)]{R} \leq C \lambda^{-1} \nrm[L^2(M)]{e^{\I\lambda \phi} \lambda^{-J} (\Delta_g - q) a_J} $.
We summarize the construction above as follows.
For our purposes we choose $J=0$.

\begin{proposition} \label{prop:r-AC22}
	Let $(M,g)$ be a simple manifold and $ M\subset M^{\mathrm{int}}_1$, where $M_1$ is also simple. Let $ y\in \partial M_1$ and $ (r,\theta)$ be the polar normal coordinates in $M_1$ with center at $y$. Let also $\nrm[L^{\infty}(M)]{q} \leq B$. Then for $\lambda \geq \lambda_0(M,g,B)$ the equation $ \mathcal{L}_{q}u=0$ in $M$ has a solution of the form
	\begin{align}
		u=e^{\I \lambda r} a+R,
	\end{align}
	where $a$ solves the transport equation $\mathcal{T}_{g,r}a=0$ defined in \eqref{eq:La-AC22}, and $R$ satisfies
	\begin{equation*}
		\nrm[L^2(M)]{R}
		\leq C \lambda^{-1} \nrm[L^2(M)]{(\Delta_g-q)a},
	\end{equation*}
	for a constant $C = C(M,g)$ independent of $\lambda$.
	The solution $a$ of $\mathcal{T}_{g,r}a=0 $ in polar normal coordinates is given by $a = |g|^{-1/4} b(\theta) $, where $b(\theta)$ is any smooth function in $\theta$.
\end{proposition}


\subsection{Proof of Theorem \ref{thm:1-AC22}} \label{subsec:Thm1-AC22}

We now give the proof of Theorem \ref{thm:1-AC22}. 
Assume $\Lambda_{q+p}^\lambda = \Lambda_q^\lambda$ for some $\lambda \geq \lambda_0$, where $\lambda_0$ shall be determined later. Since $\Lambda_q^{\lambda} = \Lambda_{q-\lambda^2}^0$ and $\Lambda_{q+p}^{\lambda} = \lambda_{q+p-\lambda^2}^0$, a standard integration by parts (see e.g.\ \cite{DKSU}*{Lemma 6.1}) implies that 
\begin{equation} \label{eq:pu0-AC22}
    \int_M p \bar{u}_1 u_2 \dif V_g
    = ((\lambda_{q+p}^\lambda - \lambda_q^\lambda) u_1, u_2)_{L^2(\partial M)}
    = 0
\end{equation}
whenever $u_1$ and $u_2$ are any $H^1(M)$ solutions of \eqref{eq:1-AC22} corresponding to $q+p$ and $q$,  respectively. We also note that the condition $\Lambda_{q+p}^\lambda = \Lambda_q^\lambda$ together with a boundary determination result imply that 
\(
p|_{\partial M} = 0.
\)
This is proved for smooth potentials e.g.\ in \cite{DKSU} and for H\"older continuous potentials in \cite{GuillarmouTzou2011}*{Proposition A.1} (recall that $q, p \in H^s(M)$ where $s > \frac{n}{2}$, so $q, p \in C^{\alpha}(M)$ for some $\alpha > 0$ by Sobolev embedding).

Due to the conditions $\nrm[H^s]{q} \leq B$, $\nrm[H^s]{p} \leq B$ stated in Theorem \ref{thm:1-AC22} and Sobolev embedding, we have $\nrm[L^{\infty}]{q} \leq CB$, $\nrm[L^{\infty}]{q+p} \leq CB$ where $C = C(M,g,s)$. By Proposition \ref{prop:r-AC22}, for $\lambda \geq \lambda_0(M,g,s,B)$ we can choose solutions $u_1$, $u_2$ having the form 
\begin{equation*} 
	\left\{\begin{aligned}
		u_1(r,\theta)
		& = e^{\I \lambda r} |g(r,\theta)|^{-1/4} b(\theta) + r_1, \\
		u_2(r,\theta)
		& = e^{\I \lambda r} |g(r,\theta)|^{-1/4} + r_2,
	\end{aligned}\right.
\end{equation*}
where $(r,\theta)$ are polar normal coordinates in $M_1$ with center at some $y \in \partial M_1$, and $b(\theta)$ shall be chosen later.
In these coordinates, $\dif V_g = |g|^{1/2} \dif r \dif \theta$.
Proposition \ref{prop:r-AC22} also gives
\begin{equation} \label{eq:rj-AC22}
	\nrm[L^2(M)]{r_1}
	\leq C \lambda^{-1} \nrm[H^2(\partial_+ S_y M_1)]{b}, \qquad
	\nrm[L^2(M)]{r_2}
	\leq C \lambda^{-1}.
\end{equation}

Substituting $u_1$, $u_2$ into \eqref{eq:pu0-AC22}, we have
\begin{align}
	0
	& = \int_M p  u_1 \bar{u}_2 \dif V_g
	= \int_M p (e^{\I \lambda r}|g|^{-1/4} b(\theta) + r_1) (e^{-\I \lambda r} |g|^{-1/4} + \bar{r}_2) \dif V_g \nonumber \\
	& = \int_M p [|g|^{-1/2} b(\theta) + e^{-\I \lambda r} |g|^{-1/4}  r_1 + e^{\I \lambda r} |g|^{-1/4} b(\theta) \bar{r}_2 +  r_1 \bar{r}_2] \dif V_g. \label{eq:pu1-AC22}
\end{align}
Recall $\dif V_g = |g|^{1/2} \dif r \dif \theta$ and $\int_0^{\tau_{M_1}(y,\theta)} p(r,\theta) \dif r = I p(y,\theta)$ where $I$ is the geodesic ray transform on $M_1$. Here we assume that $p$ is extended by zero to $M_1$. Thus we also have
\begin{equation} \label{eq:pb-AC22}
	\int_M p(r,\theta) b(\theta) |g(r,\theta)|^{-1/2} \dif V_g
	= \int_{\partial_{+} S_{y}M_1} Ip(y,\theta) b(\theta) \dif \theta.
\end{equation}
From $p \in \mathcal{A}_s(B)$ where $s >  \frac{n}{2} $ and from the Sobolev embedding we can conclude that $\nrm[L^\infty(M)]{p} \leq C \nrm[H^s(M)]{p} \leq CB \nrm[L^2(M)]{p}$.
Therefore, from \eqref{eq:rj-AC22}, \eqref{eq:pu1-AC22} and \eqref{eq:pb-AC22} it follows that, with implied constants depending on $B$, 
\begin{align*}
	\big| \int_{\partial_{+} S_{y}M_1} Ip(y,\theta) b(\theta) \dif \theta \big|
	& \lesssim \nrm[L^2(M)]{p} \nrm[L^2(M)]{r_1} + \nrm[L^\infty(M)]{p} \nrm[L^2(M)]{b} \nrm[L^2(M)]{r_2} \\
	& \qquad + \nrm[L^\infty(M)]{p} \nrm[L^2(M)]{r_1} \nrm[L^2(M)]{r_2} \\
	& \lesssim \nrm[L^2(M)]{p} \left[\nrm[L^2(M)]{r_1} + (\nrm[L^2(\partial_+ S_y M_1)]{b} + \nrm[L^2(M)]{r_1}) \nrm[L^2(M)]{r_2}\right] \\
	& \lesssim \lambda^{-1} \nrm[L^2(M)]{p} \nrm[H^2(\partial_+ S_y M_1)]{b},
\end{align*}
where we used \eqref{eq:rj-AC22}.
This estimate further gives
\begin{align*}
	\big| \int_{\partial_{+}SM_1} Ip(y,\theta) b(\theta) \dif (\partial SM) \big|
	& \leq \int_{\partial M_1} \big| \int_{\partial_{+} S_{y}M_1} Ip(y,\theta) b(\theta) \dif \theta \big| \dif y \nonumber \\
	& \lesssim \lambda^{-1} \nrm[L^2(M)]{p} \int_{\partial M_1} \nrm[H^2(\partial_{+} S_y M_1)]{b} \dif y. 
\end{align*}
Note that the function $ b(\theta)$ depends on $y$.

Choosing $b(\theta) = \langle \nu_y,\theta\rangle I(I^*Ip)$, inserting this in the above inequality, and using the Santal\'o formula (Lemma \ref{lem:San1-AC22}) and boundedness of $I$ and $I^* I$ (Lemmas \ref{lemma_i_bounded} and \ref{lemma_simple_normal}),  we obtain
\begin{align} 
	\lVert I^*Ip \rVert^2 _{L^2(M_1)}
	&\lesssim \lambda^{-1} \nrm[L^2(M)]{p} \nrm[H^2(\partial_{+} SM_1)]{I(I^*Ip)} 
	\lesssim \lambda^{-1} \nrm[L^2(M)]{p} \nrm[H^2(M_1)]{I^*Ip}  \notag \\
	&\lesssim \lambda^{-1} \nrm[H^1(M)]{p}^2. \label{eq:IIpa-AC22}
\end{align}
Here we also used the condition $p|_{\partial M} = 0$, which allows us to consider $p$ as a function in $H^1(M_1)$ with support in $M$. Using the interpolation $\nrm[H^1(M_1)]{f}^2 \leq C \nrm[L^2(M_1)]{f} \nrm[H^2(M_1)]{f}$ \cite{Taylor_book}*{Proposition 3.1} between Sobolev spaces, we see $\nrm[H^1(M_1)]{I^*Ip}^2$ can be bounded by the product of $\nrm[L^2(M_1)]{I^*Ip}$ and $\nrm[H^2(M_1)]{I^*Ip}$.
The $L^2$ norm of $I^*Ip$ can be estimated from \eqref{eq:IIpa-AC22}, while the $H^2$ norm of $I^*Ip$ can be estimated by using the continuity of $I^* I$, thus
\begin{equation*}
	\nrm[H^1(M_1)]{I^*Ip}^2
	\leq C \lVert I^*Ip \rVert _{L^2(M_1)} \lVert I^*Ip \rVert _{H^2(M_1)}
	\lesssim \lambda^{-1/2} \nrm[H^1(M_1)]{p}^2.
\end{equation*}
Recall that $p \in \mathcal{A}_s(B)$ with $s > n/2 \geq 1$, so $\nrm[H^1(M_1)]{p} \lesssim \nrm[L^2(M_1)]{p}$.
This together with the inequality above gives
\begin{equation*}
	\nrm[H^1(M_1)]{I^*Ip}^2
	\lesssim \lambda^{-1/2} \nrm[L^2(M_1)]{p}^2.
\end{equation*}
Because $(M,g)$ is assumed to be a simple manifold, by Lemma \ref{lemma_simple_normal} we know that $I^* I$ is stably invertible, namely, $\nrm[L^2(M_1)]{p} \leq C \nrm[H^1(M_1)]{I^*Ip}$.
Combining this with the last displayed equation above, we arrive at
\begin{align*}
\nrm[L^2(M_1)]{p}^2
\leq C_B \lambda^{-1/2} \nrm[L^2(M_1)]{p}^2
\ \implies \
(1 - C_B^{1/2} \lambda^{-1/4}) \nrm[L^2(M_1)]{p}
\leq 0.
\end{align*}
By setting $\lambda_0(M,g,B) := 2 C_B^2$ and choosing $\lambda \geq \lambda_0(M,g,B)$, we can conclude from the above that $\nrm[L^2(M)]{p} \leq 0$, so $p = 0$ in $M$.
This completes the proof of Theorem \ref{thm:1-AC22}.

\section{Proof of Theorem \ref{thm:2-AC22}} \label{sec:Thm2-AC22}


In this section, we present the proof of Theorem \ref{thm:2-AC22}.
As in the proof of Theorem \ref{thm:1-AC22}, the assumption $\Lambda_{q+p}^{\lambda} = \Lambda_{q}^{\lambda}$ leads to the integral identity 
\begin{equation} \label{eq:pu12-AC22}
	\int_M p u_1 \overline{u}_2 \dif V_g = 0.
\end{equation}
 Here, $u_1$ and $u_2$ solve \eqref{eq:1-AC22} with potentials being $q+p$ and $q$, respectively. We will choose $u_1$ and $u_2$ to be Gaussian beam quasimodes concentrated near a geodesic $\gamma$ based on Theorem \ref{thm:gbu-AC22}.
According to Theorem \ref{thm:gbu-AC22}, $u_j~(j=1,2)$ can be represented as $u_j = v + r_j$, where $v$ is the leading term and $r_j$ are the corresponding remainder terms.
Here $v$ is the leading term of both $u_1$ and $u_2$.
Note that $u_1$ and $u_2$ have the same leading term because the leading term depends only on the metric and the geodesic. Since the leading term $v$ concentrate on the geodesic $ \gamma$, the term $ \int_{M} p\, |v|^2$  can be estimated using Theorem \ref{thm:gbuc-AC22}. This implies 
\begin{align} \label{estimate_of_Ip}
   \big| \int_{M}  p\, |v|^2 \dif V_g - Ip(\gamma) \big| \leq C \nrm[C^1(M)]{p} h^{1/2},
\end{align}
where $Ip$ stands for the geodesic ray transform of $p$.
Next substituting $u_j=v+r_j$ into \eqref{eq:pu12-AC22} for $j=1,2$, we obtain
\begin{equation} \label{eq:pu13-AC22}
    \int_M p |v|^2 \dif V_g
    = -\int_M p (\overline{v} r_1 + v \overline{r}_2 + r_1 \overline{r}_2) \dif V_g,
\end{equation}
where $r_1$ and $r_2$ are error terms that can be estimated using the following result.

\begin{lemma} \label{lem:rEst-AC22}
    Let $r_1$ and $r_2$ be given as above.
    There exists a constant $C$ uniformly with respect to $\gamma$ such that
    \(
    \norm{r_j}_{L^2(M)} \leq C \lambda^{-1}
    \)
    for $\lambda \geq \lambda_0(M,g,B)$, $j=1,2$.
\end{lemma}

\begin{proof}
    We only give the proof for $r_1$, and that of $r_2$ is similar.
    By Proposition \ref{prop:resol-AC22} we have
    \(
    \norm{r_1}_{L^2(M)}
    \leq C \lambda^{-1} \nrm[L^2(M)]{\mathcal L_{q,\lambda} v}.
    \)
    The quantity $\mathcal L_{q,\lambda} v$ can be expressed in terms of $h$ as $\lambda^2 (-h^2 \Delta_g - 1) v + q v$.
    From Theorem \ref{thm:gbu-AC22} we can bound the $L^2$-norm of both $(-h^2 \Delta_g - 1) v$ and $v$, thus
    \begin{align*}
        \norm{r_1}_{L^2(M)}
        & \leq C_1 \lambda^{-1} (\nrm[L^2(M)]{\lambda^2 (-h^2 \Delta_g - 1) v} + \nrm[L^2(M)]{q v}) \\
        & \leq C_1 \lambda^{-1} (C_2 \lambda^2 h^K + C_2 \nrm[L^\infty(M)]{q})
        \leq C_3 \lambda^{-1}.
    \end{align*}
    The constant $C_1$ comes from the resolvent estimate given in Proposition \ref{prop:resol-AC22}, so it does not depend on the choice of the geodesic $\gamma$.
    The uniformity of $C_2$ with respect to $\gamma$ is guaranteed by Theorem \ref{thm:gbuc-AC22}, respectively.
    Therefore, $C_3$ is uniform with respect to the choice of the geodesic $\gamma$.
    The proof is done.
\end{proof}

The combination of \eqref{estimate_of_Ip}, \eqref{eq:pu13-AC22} and Lemma \ref{lem:rEst-AC22} entails
\begin{align*}
    |Ip(\gamma)|
    & \le C   h^{1/2} \nrm[C^1(M)]{p} + \big |\int_M p (\overline{v} r_1 + v \overline{r}_2 + r_1 \overline{r}_2) \dif V_g \big| \\
    &\le  C h^{1/2} \nrm[C^1(M)]{p} + \left[\nrm[L^2(M)]{v} \mathcal O(\lambda^{-1}) + \mathcal O(\lambda^{-1}) \right]\nrm[L^\infty(M)]{p}.
\end{align*}
Since $p \in \mathcal A_s(B)$ and $s > 1 + \frac n 2$, we have $\nrm[L^\infty(M)]{p} \leq \nrm[C^1(M)]{p} \lesssim \nrm[H^s(M)]{p} \leq B \nrm[L^2(M)]{p}$.
By Theorem \ref{thm:gbu-AC22} we also have that $ \nrm[L^2(M)]{v} \leq C$. The combination of these with the above inequality imply
\begin{equation*}
	|Ip(\gamma)|
	\leq C h^{1/2}\,\nrm[L^2(M)]{p}.
\end{equation*}
This further gives (writing $\lambda=h^{-1}$)
\begin{equation} \label{eq:Ipla-AC22}
	\nrm[L^2(\partial_+ S M_1)]{Ip}
	\leq B \mathcal O(\lambda^{-1/2}) \nrm[L^2(M)]{p} .
\end{equation}
Here we considered $M$ to be embedded into a slightly larger manifold $M_1$ and extended $p$ by zero to $M_1$ as in Section \ref{subsec:Thm1-AC22}.
Then by using:
(i) the stable invertibility of $I$ with respect to the $L^2(M_1)$ and $H^1(\partial_+ S M_1)$ norms (cf.~\eqref{eq:stable-AC22}),
(ii) the interpolation $\nrm[H^1(\partial_+ S M_1)]{\varphi}^2 \leq \nrm[L^2(\partial_+ S M_1)]{\varphi} \nrm[H^2(\partial_+ S M_1)]{\varphi}$,
(iii) the estimate \eqref{eq:Ipla-AC22},
(iv) the continuity of $I \colon H^2(M_1) \to H^2(\partial_+ S M_1)$ (cf.\ \eqref{eq:stable2-AC22}),
(v) and the assumption $p \in \mathcal A_s(B)$
sequentially, we can make the following derivation,
\begin{align*}
	\nrm[L^2(M_1)]{p}^2
	& \leq \nrm[H^1(\partial_+ S M_1)]{Ip}^2
	\leq \nrm[L^2(\partial_+ S M_1)]{Ip} \nrm[H^2(\partial_+ S M_1)]{Ip} \\
	& \leq \mathcal O(\lambda^{-1/2}) \nrm[L^2(M_1)]{p} \nrm[H^2(M_1)]{p}
	\leq \mathcal O(\lambda^{-1/2}) \nrm[L^2(M_1)]{p}^2, \quad \forall \lambda \geq \lambda_{M,g,B}.
\end{align*}
The implicit constant also depends on $B$. Therefore, there exists a constant $C$ such that
\begin{equation} \label{eq:Ipl-AC22}
	(1 - C \lambda^{-1/4}) \nrm[L^2(M_1)]{p}
	\leq 0, \quad \forall \lambda \geq \lambda_{M,g,B}.
\end{equation}
Hence, we conclude $p = 0$ in $M$. This concludes the proof of Theorem \ref{thm:2-AC22}.

\begin{remark} \label{rem:sta1-AC22}
    When $\Lambda_{q+p}^\lambda \neq \Lambda_{q}^\lambda$, by \eqref{eq:pu0-AC22}
    and the arguments in Section \ref{sec:Thm2-AC22} we obtain
    \begin{equation*}
        Ip(\gamma)
        = ((\Lambda_{q+p}^\lambda - \Lambda_{q}^\lambda) u_1, u_2)_{L^2(\partial M)} + \nrm[L^2(M)]{p} \mathcal O(\lambda^{-a})
    \end{equation*}
    for any $a \in (1/3,1/2)$, where we used the assumption $p \in \mathcal A_s(B)$ for $s > n/2 + 1$.
    We denote $\epsilon := \nrm[H^{1/2}(\partial M) \to H^{-1/2}(\partial M)]{\Lambda_{q+p}^\lambda - \Lambda_{q}^\lambda}$, and $u_j = v + r_j~(j=1,2)$ as in Section \ref{sec:Thm2-AC22}, then
    \begin{align*}
        \nrm[L^2(M_1)]{p}
        & \lesssim \nrm[H^1(\partial_+ S M_1)]{Ip}
        \lesssim \epsilon \nrm[H^{1/2}(\partial M)]{v + r_1} \nrm[H^{1/2}(\partial M)]{v + r_2} + \nrm[L^2(M)]{p} \lambda^{-a} \nonumber \\
        & \lesssim \epsilon (\nrm[H^1(M)]{v}^2 + \nrm[H^1(M)]{v} \nrm[H^1(M)]{r} + \nrm[H^1(M)]{r}^2) + \nrm[L^2(M)]{p} \lambda^{-a}.
    \end{align*}
    In the derivation above, because both $r_1$ and $r_2$ follow the same estimate with respect to $\lambda$ and $\delta$, we don't distinguish them by simply represent both of them as $r$.
    Absorbing the $\lambda^{-a} \nrm[L^2(M)]{p}$ term by the left-hand-side, we finally obtain
    \begin{align}
        (1 - \lambda^{-a}) \nrm[L^2(M_1)]{p}
        & \leq C \nrm[H^{1/2}(\partial M) \to H^{-1/2}(\partial M)]{\Lambda_{q+p}^\lambda - \Lambda_{q}^\lambda} \nonumber \\
        & \quad \times (\nrm[H^1(M)]{v}^2 + \nrm[H^1(M)]{v} \nrm[H^1(M)]{r} + \nrm[H^1(M)]{r}^2), \label{eq:sta-AC22}
    \end{align}
    The $L^2$ norm of $v$ and $r$ can be  investigated using the estimates given in  Theorem \ref{thm:gbu-AC22}.
    To obtain their $H^1$ norm, we need to analyze their gradients, which shall give certain growth of order $\lambda^{b_1} \delta^{b_2}$ for certain $b_1$, $b_2 \in \R$.
    We defer this to future works.
\end{remark}

\begin{remark} \label{rem:sta2-AC22}
    Our method can also be utilized to obtain stability estimates in certain Sobolev spaces. However, in this case one can obtain H\"older-type stability estimates by examining the difference between the DN maps at large frequency. This is consistent with the phenomenon of improved stability for high frequency Schr\"odinger operators on $\Rn$, which has already been investigated  in the literature; see for instance \cite{Isakov_increasing_stability_2014} and  the references therein.
\end{remark}

\section{Gaussian beams with uniform constants} \label{sec:UC-AC22}

In this section we give an invariant construction of Gaussian beam quasimodes with uniform bounds for the underlying constants. Let $(M,g)$ be a compact manifold with smooth boundary. For any $(x,v) \in \partial_+ SM$ let $\gamma_{x,v}: [0,\tau(x,v)] \to M$ be the maximally extended unit speed geodesic starting at $x$ in direction $v$. We allow the manifold to have trapped geodesics (i.e.\ $\tau(x,v)$ may be $+\infty$ for some $(x,v)$), but below we will only work with $(x,v) \in \mathcal{G}_T$ where 
\[
\mathcal{G}_T = \{ (x,v) \in \partial_+ SM \,:\, \tau(x,v) \leq T \}.
\]
The following result states the existence of Gaussian beam quasimodes concentrating near $\gamma_{x,v}$ with uniform bounds over $(x,v) \in \mathcal{G}_T$. Recall that $I$ denotes the geodesic X-ray transform on $(M,g)$.

\begin{theorem} \label{thm:gbuc-AC22}
    Let $(M,g)$ be a compact oriented manifold with smooth boundary. Fix $T > 0$ and $k, K \geq 0$. There is a constant $C = C(M,g,T,k,K) > 0$ such that for any $(x,v) \in \mathcal{G}_T$ and $h \in (0,1)$, there is $u = u_{x,v,h} \in C^{\infty}(M)$ satisfying 
    \begin{gather} \label{eq:conc-AC22}
    \left\lvert \int_M |u|^2 \varphi \dif V_g - I\varphi(x,v) \right\rvert \leq C \norm{\varphi}_{C^1(M)} h^{1/2}, \\
    \norm{(-h^2 \Delta_g - 1) u}_{H^k(M)} \leq C h^K,\notag
    \end{gather}
    uniformly over $0 < h < 1$ and $\varphi \in C^1(M)$.    
\end{theorem}

Theorem \ref{thm:gbuc-AC22} is sufficient for proving Theorem \ref{thm:2-AC22}. For later purposes, we also state a result that describes the form of $u_{x,v,h}$ more precisely and involves normalization in $L^p$. Below, for a tensor $A$ at $x$ and a subspace $F$ of $T_x M$ we write $A|_{F}$ for the multilinear form that only acts on vectors in $F$.

\begin{theorem} \label{thm:gbu-AC22}
Let $(M,g)$ be a compact oriented manifold with smooth boundary. Fix constants $T > 0$, $p \in [1,\infty)$, $k \geq 0$, and $K \geq 0$. There is a constant $C = C(M,g,T,p,k,K) > 0$ such that for any $(x,v) \in \mathcal{G}_T$ and $h \in (0,1)$, there is $u = u_{x,v,h} \in C^{\infty}(M)$ associated with $\gamma = \gamma_{x,v}$ and satisfying 
\begin{gather*}
\norm{u}_{L^p(M)} \leq C, \\
\norm{(-h^2 \Delta_g - 1)u}_{W^{k,p}(M)} \leq C h^K, \\
\mathrm{supp}(u) \subset \{ y \in M \,:\, \mathrm{dist}_g \big( y, \gamma([0,\tau(x,v)]) \big) \leq C^{-1} \},
\end{gather*}
uniformly over all $0 < h < 1$.

There is also a symmetric complex $(1,1)$-tensor $H(t) = H_{x,v}(t)$ on $T_{\gamma(t)} M$, depending smoothly on $t \in [0,\tau(x,v)]$ and satisfying 
\[
\mathrm{Im}(H(t)^{\flat}) \geq 0, \qquad \mathrm{Im}(H(t)^{\flat})|_{\dot{\gamma}(t)^{\perp}} \geq C^{-1} g|_{\dot{\gamma}(t)^{\perp}},
\]
such that $u = u_{x,v,h}$ has the following form. If $x_0 \in \gamma([0,\tau(x,v)])$ and if $t_1 < \ldots < t_{N_p}$ are the times in $[0,\tau(x,v)]$ when $\gamma(t_l) = x_0$, then in a small neighborhood $U$ of $x_0$ we have 
\[
u|_U = u^{(1)} + \ldots + u^{(N_p)}
\]
where each $u^{(l)}$ satisfies 
\[
u^{(l)}|_U = h^{-\frac{n-1}{2p}} e^{\I \Phi^{(l)}/h} (a_0^{(l)} + h a_1^{(l)} + \ldots + h^N a_N^{(l)}) \rho.
\]
Here $N = N(M,g,T,p,k,K)$, and $\rho$ is a smooth cutoff function supported near $\gamma|_{[t_l-\varepsilon,t_l+\varepsilon]}$. The phase $\Phi = \Phi^{(l)}$ is independent of $h$ and satisfies for $t$ near $t_l$ 
\[
\Phi(\gamma(t)) = t, \quad \nabla \Phi(\gamma(t)) = \dot{\gamma}(t), \quad \nabla^2 \Phi(\gamma(t)) = H(t)^{\flat}, \qquad \norm{\Phi}_{C^k(\overline{U})} \leq C.
\]
The amplitudes $a_j^{(l)}$ are independent of $h$, and for $t$ near $t_l$ one has 
\[
a_0^{(l)}(\gamma(t)) = \exp \left[ -\frac{1}{2} \int_0^t \mathrm{tr}_g(H(s)) \dif s \right], \qquad \norm{a_j^{(l)}}_{C^k(\overline{U})} \leq C.
\]
If $p=2$, then $c_n u$ satisfies the conditions in Theorem \ref{thm:gbuc-AC22} with $\displaystyle c_n = (\int_{\R^{n-1}} e^{-\abs{y}^2} \dif y)^{-1/2}$.
\end{theorem}

\begin{remark}
We note that if $(M,g)$ is compact and nontrapping, i.e.\ $\tau(x,v) < \infty$ for all $(x,v) \in \partial_+ SM$, then $\mathcal{G}_T = \partial_+ SM$ for sufficiently large $T$. If $\partial M$ is strictly convex, this follows from the continuity of $\tau$. In general one can argue as follows: suppose $\tau(x_j,v_j) \to \infty$ for some sequence $(x_j, v_j) \in \partial_+ SM$. After choosing a subsequence, we have $(x_j, v_j) \to (x,v) \in \overline{\partial_+ SM}$. Since $\tau$ is upper semicontinuous, 
\[
\lim \tau(x_j, v_j) \leq \tau(x,v).
\]
This is a contradiction, since $\tau(x,v) < +\infty$ by the nontrapping condition.
\end{remark}



We will prove Theorems \ref{thm:gbuc-AC22} and \ref{thm:gbu-AC22} in two parts: first in the case of non-self-intersecting geodesics, and then in the self-intersecting case.

\subsection{No self-intersection case} \label{subsec:UCno-AC22}

Let $(M,g)$ be a compact oriented manifold with smooth boundary. Fix $T > 0$ and define
\[
\tilde{\mathcal{G}}_T := \{ (x,v) \in \mathcal{G}_T \,:\, \text{$\gamma_{x,v}$ does not self-intersect} \}.
\]
We will prove Theorem \ref{thm:gbu-AC22} for all $(x,v) \in \tilde{\mathcal{G}}_T$, and in Section \ref{subsec:UCse-AC22} we reduce the general case $(x,v) \in \mathcal{G}_T$ to this case.

Let $(x,v) \in \tilde{\mathcal{G}}_T$ and let $\gamma = \gamma_{x,v}$. We look for $u = u_{x,v,h}$ in the form 
\begin{equation} \label{u_nsi_form}
u = h^{-\frac{n-1}{2p}} e^{i \Phi/h} (a_0 + h a_1 + \ldots + h^N a_N) \rho
\end{equation}
where $\rho$ is a suitable cutoff function. The functions $\Phi$ and $a_j$ will be constructed in an invariant fashion, but in order to do this we need some preparations.


We assume that $(M,g)$ is embedded in a closed manifold $(S,g)$ of the same dimension.
We will also consider a cutoff function $\chi \in C^{\infty}_c(\R)$ with $0 \leq \chi \leq 1$, $\chi(x) = 1$ for $|x| \leq 1/2$, and $\chi(x) = 0$ for $|x| \geq 1$. We consider $(S,g)$ and $\chi$ to be fixed once and for all. The constructions and constants below will depend on the choice of $(S,g)$ and $\chi$ but we will not write out this dependence.

Define
\begin{equation} \label{eq:Udel-AC22}
    U(\gamma_{x,v}, \delta) = \{ (t, y) \,:\, t \in (-\delta,\tau(x,v)+\delta), \ y \perp \dot{\gamma}_{x,v}(t), \ |y| < \delta \},
\end{equation}
and let $\delta_{x,v}$ be the supremum of $\delta \geq 0$ such that the map $F_{x,v}: U(\gamma_{x,v},\delta) \to S, \ (t,y) \mapsto \exp_{\gamma_{x,v}(t)}(y)$ is a diffeomorphism onto its image. By the inverse function theorem, since $dF_{x,v}|_{(t,0)}$ corresponds to the identity map, there is a positive lower bound for $\delta_{x,v}$ that depends on the $C^2$ norm of $F_{x,v}$ on $U(\gamma_{x,v},1)$ (see e.g.\ \cite{PSU_book}*{Lemma 11.2.6} for a similar argument). Hence, the constant
\[
\mathrm{inj}_F(M,g)
:= \inf_{(x,v) \in \mathcal{G}_T} \delta_{x,v}
\]
is positive due to the compactness of $M$.
Below we fix $\delta = \mathrm{inj}_F(M,g)/2$. 

The phase function $\Phi$ is specified in the following lemma.

\begin{lemma} \label{lem:a1-AC22}
Let $N \geq 0$ be an integer. For any $\gamma = \gamma_{x,v}$ with $(x,v) \in \mathcal{G}_T$ there is a unique function $\Phi = \Phi_{x,v,N} \in C^{\infty}(M; \mathbb C)$ satisfying the following conditions for any $t \in [-\delta,\tau(x,v)+\delta]$:
\begin{enumerate}
    \item[(a)] 
    $\nabla^j(\langle \df \Phi, \df \Phi \rangle-1)(\gamma(t)) = 0 \text{ for $0 \leq j \leq N+2$}$,
    
    \item[(b)]
    $\Phi(\gamma(t)) = t, \ \df \Phi(\gamma(t)) = (\dot \gamma(t))^\sharp$,
    
    \item[(c)]
     $\nabla^2 \Phi(\gamma(0))|_{\dot{\gamma}(0)^{\perp}} = \I g|_{\dot{\gamma}(0)^{\perp}}, \ \nabla^j \Phi(\gamma(0))|_{\dot{\gamma}(0)^{\perp}} = 0 \text{ for $3 \leq j \leq N$}$,
    
    \item[(d)]
    $\Phi(\exp_{\gamma(t)}(y)) = \chi(|y|/\delta) \sum_{j=0}^N \frac{\nabla^j \Phi|_{\gamma(t)}(y, \ldots, y)}{j!}$ in $U(\gamma, \delta)$,
    
    \item[(e)]
    $\Phi = 0$ outside $U(\gamma,\delta)$.
\end{enumerate}
There are constants $C, c > 0$ only depending on $(M,g)$, $T$ and $N$ such that
\begin{equation} \label{eq:aq2-AC22}
    \norm{\Phi}_{C^N(M)} \leq C, \quad \mathrm{Im}(\nabla^2 \Phi)|_{\dot{\gamma}(t)^{\perp}} \geq c g|_{\dot{\gamma}(t)^{\perp}}
\end{equation}
whenever $t \in [-\delta, \tau(x,v) + \delta]$.
\end{lemma}


Define the transport operator $L$ by 
\[
Lv := \frac 1 \I (2 \langle \df \Phi, \df v \rangle + (\Delta_g \Phi)v).
\]
The amplitudes $a_r$ are given as follows.

\begin{lemma} \label{lem:a2-AC22}
Let $N \geq 0$ be an integer. For any $(x,v) \in \mathcal{G}_T$ there are unique functions $a_r = a_{r,x,v,N} \in C^{\infty}(M; \mathbb C)$ with $0 \leq r \leq N$ satisfying the following conditions for any $t \in [-\delta,\tau(x,v)+\delta]$:
\begin{enumerate}
    \item[(a)] 
    $\nabla^j(La_0)(\gamma(t)) = 0 \text{ for $0 \leq j \leq N$}$,
    
    \item[(b)]
    $a_0(\gamma(0)) = 1, \ \nabla^j a_0(\gamma(0))|_{\dot{\gamma}(0)^{\perp}} = 0 \text{ for $1 \leq j \leq N$}$,
    
    \item[(c)] 
    $\nabla^j(La_r - \Delta_g a_{r-1})(\gamma(t)) = 0 \text{ for $0 \leq j \leq N$ and $1 \leq r \leq N$}$,
    
    \item[(d)]
    $a_r(\gamma(0)) = 0$, $\nabla^j a_r(\gamma(0))|_{\dot{\gamma}(0)^{\perp}} = 0 \text{ for $1 \leq r \leq N$ and $1 \leq j \leq N$}$,
    
    \item[(e)]
    $a_r(\exp_{\gamma(t)}(y)) = \chi(|y|/\delta) \sum_{j=0}^N \frac{\nabla^j a_r|_{\gamma(t)}(y, \ldots, y)}{j!}$ in $U(\gamma, \delta)$,
    
    \item[(f)]
    $a_r = 0$ outside $U(\gamma,\delta)$.
\end{enumerate}
There is a  constant $C> 0$ only depending on $(M,g)$, $T$ and $N$ such that 
\[
\norm{a_r}_{C^N(M)} \leq C.
\]
Moreover, if $H(t) = \nabla^2 \Phi(\gamma(t))^{\sharp}$, one has 
\begin{align}\label{principle_term}
    a_0(\gamma(t)) = \exp \left[ -\frac{1}{2} \int_0^t \mathrm{tr}_g(H(s)) \dif s \right].
\end{align}
\end{lemma}

\begin{proof}[Proof of Theorem \ref{thm:gbu-AC22}]
    We now prove Theorem \ref{thm:gbu-AC22} for $(x,v) \in \tilde{\mathcal{G}}_T$. By Lemmas \ref{lem:a1-AC22} and \ref{lem:a2-AC22}, the $C^N(M)$ norms of $\Phi$ and $a$ are uniformly bounded over $(x,v) \in \tilde{\mathcal{G}}_T$.
    Moreover, we have the estimate $\mathrm{Im} \big( \Phi(\exp_{\gamma(t)}(y)) \big) \geq c |y|^2 - C |y|^3$ where $c, C > 0$ are uniform over $(x,v)$. We now choose $\delta_1 = \delta_1(C,c) < \delta$ so that 
\begin{align}\label{Gaussian_estimate}
   \mathrm{Im} \big( \Phi(\exp_{\gamma(t)}(y)) \big) \geq c |y|^2/2, \quad |y| < \delta_1. 
\end{align}
The function $\rho$ in Theorem \ref{thm:gbu-AC22} is chosen as $\rho(t,y) = \chi(|y|/\delta_1)$. Using the above facts, all constants below will be uniform over $(x,v) \in \tilde{\mathcal{G}}_T$.

We now compute the $L^p(M)$ norm of $u$. Due to the presence of the cutoff function $\rho$, it is enough to calculate the $L^p$ norm $u$ in $ U(\gamma, \delta)$. This  along with \eqref{Gaussian_estimate} entails
\begin{equation*}
    \nrm[L^p(M)]{u}^p
    \lesssim \int_{-\delta}^{\tau+\delta} \int_{ |y|< \delta_1} \frac 1 {h^{(n-1)/2}} |e^{-c|y|^2/2h}|^p \,|g|^{1/2}\dif t \dif y
    \lesssim (2\delta +\tau) \int_{\R^{n-1}} e^{-c p|z|^2/2} \dif z.
\end{equation*}
This shows that $\nrm[L^p(M)]{u} \lesssim 1$ uniformly over $(x,v) \in \Tilde{\mathcal{G}}_T$.

Let us then denote $f = (-h^2\Delta_g-1) (e^{\I \Phi/h} a)$ where $a = (a_0 + h a_1 + \ldots + h^N a_N) \chi(|y|/\delta_1)$.
A direct computation shows that
\begin{equation*}
    f = \frac {e^{\I \Phi/h}} {h^{(n-1)/(2p)}} \big[ (|\df \Phi|_{g}^2-1) a \chi(|y|/\delta_1) + hf_1 + \ldots + h^{N-1} f_{N-1}  - h^{N} \Delta_g(  \chi(y/\delta_1) a_N) \big],
\end{equation*}
where $f_j$ are smooth functions vanishing of order $N$ on $\gamma$ for $1\le j\le N-1$, due to the properties of $a_k$ in Lemma  \ref{lem:a2-AC22}. Also note that each $f_j$ contains two terms: one involves the operator $L$ acting on $a_k$, and another term is involving derivatives of the cutoff function $\chi(\abs{y}/\delta_1)$. The term involving derivatives of $\chi(\abs{y}/\delta_1)$ is $\mathcal{O}(h^K)$ for all $K$, due to the Gaussian nature of $e^{\I \Phi/h}$. Thus we ignore this term when we compute $\nrm[L^p(M)]{f}$. 
Observe that 
\begin{equation*}
    |e^{\I \Phi/h}|
    \leq e^{-c|y|^2/2h} \ \ \text{in} \ \ \mathrm{supp}(f).
\end{equation*} 
This  implies 
\begin{align} \label{estimates_of_f}
    \abs{f} \lesssim \frac{1}{h^{(n-1)/(2p)}} e^{-c|y|^2/2h} (|y|^{N+1}+ h^N).
\end{align}

It is enough to estimate $\nrm[L^p(M)]{f}$ in the neighbourhood $\{(t,y) : -\delta < t < \tau + \delta, |y|< \delta_1\}$ where $f$ is supported. Here $\tau = \tau(x,v)$. Then from \eqref{estimates_of_f} we obtain
\begin{align*}
    \nrm[L^p(M)]{f}^p
    & \lesssim \frac{1}{h^{(n-1)/2}} \int_{-\delta}^{\tau+\delta} \int_{ |y|< \delta_1} |f|^p \dif y \dif t
    \lesssim \frac{(\tau+2\delta) }{h^{(n-1)/2}} \int_{ |y|< \delta} \abs{e^{-c|y|^2/2h} (|y|^{N+1} + h^N)}^p \dif y \\
    & \lesssim   \frac{(\tau+2\delta) }{h^{(n-1)/2}}\int_{\R^{n-1}} \abs{e^{-c|z|^2/2} \,  (|z|^{N+1} h^{(N+1)/2} + h^N) }^p \, h^{\frac{n-1}{2}} \dif z \\
    & \lesssim (\tau+2\delta)  \big[ h^{\frac {(N+1)p} 2} \int_{\R^{n-1}} e^{-cp|z|^2/2} |z|^{(N+1)p} \dif z + h^{Np} \int_{\R^{n-1}} e^{-cp|z|^2/2} \dif z \big] \\
    & \lesssim (\tau+2\delta) h^{\frac {(N+1)p} 2}.
\end{align*}
Since $\tau(x,v) \leq T$ uniformly over $(x,v) \in \Tilde{\mathcal{G}}_T$ and since $\delta$ is fixed, we conclude
\begin{align*}
  \nrm[L^p(M)]{f} \lesssim h^{\frac{N+1}{2}}.
\end{align*}
In order to estimate  the $L^p(M)$ norm of higher order derivatives of $f$, we apply $\nabla_g$ on $f$ and observe that 
\begin{align*}
   h^{(n-1)/(2p)} \nabla_g f
   = \ & \frac{ \I \nabla_g \Phi}{h} e^{\I \Phi/h} \big[ (|\df \Phi|_{g}^2-1) a \chi(|y|/\delta_1) + \sum_{j=1}^{N-1} h^j f_j - h^{N} \Delta_g(  \chi(y/\delta_1) a_N) \big] \\ & + e^{\I \Phi/h}  \nabla_g\big[ (|\df \Phi|_{g}^2-1) a \chi(|y|/\delta_1) +\sum_{j=1}^{N-1} h^j f_j - h^{N} \Delta_g(  \chi(y/\delta_1) a_N) \big].
\end{align*} 
Next, utilizing similar arguments as above, we obtain
\begin{align*}
    \abs{\nabla_g f}\lesssim \frac 1 {h^{(n-1)/(2p)}} e^{-c|y|^2/2h} \big[ \frac 1 h (|y|^{N+1} + h^N) + (|y|^N + h^N) \big].
\end{align*}
This further entails
\begin{align*}
    \nrm[L^p(M)]{\nabla_g f}^p
    & \lesssim \frac{(\tau+2\delta) }{h^{(n-1)/2}} \int_{ |y|< \delta} \abs{e^{-c|y|^2/2h} \big[ \frac 1 h (|y|^{N+1} + h^N) + (|y|^N + h^N) \big]}^p \dif y \\
    & \lesssim   \frac{(\tau+2\delta) }{h^{(n-1)/2}}\int_{\R^{n-1}} \abs{e^{-c|z|^2/2} \, \big[ \frac 1 h (h^{\frac {N+1} 2} |z|^{N+1} + h^N) + (h^{\frac N 2} |z|^N + h^N) \big]}^p \, h^{\frac{n-1}{2}} \dif z \\
    & \lesssim (\tau+2\delta) \int_{\R^{n-1}} e^{-cp|z|^2/2} \big[ h^{\frac {(N-1)p} 2} |z|^{(N+1)p} + h^{\frac {Np} 2} |z|^{Np} \dif z + h^{(N-1)p} \big] \dif z \\
    & \lesssim (\tau+2\delta) (h^{\frac {(N-1)p} 2} + h^{\frac {Np} 2} + h^{(N-1)p})
    \lesssim (\tau+2\delta) h^{\frac {(N-1)p} 2},
\end{align*}
namely,
\begin{align*}
    \nrm[L^p(M)]{\nabla_g f}
    \lesssim (\tau+2\delta)^{1/p}\, h^{\frac {(N-1)} 2}.
\end{align*}
Similarly, one can obtain the following bound $\nrm[L^p(M)]{\nabla_g^k f}^p \lesssim (\tau+2\delta) h^{\frac {(N+1)p} 2} h^{-k}$ for the higher order derivatives of $f$. After choosing $N = N(K,k,p)$ in a suitable way, this gives the required bound for the $W^{k,p}(M)$ norm of $ (-h^2\Delta_g-1) u$.

The condition for $\mathrm{supp}(u)$ follows from the presence of the cutoff function $\rho$. Writing $H(t) = \nabla^2 \Phi(\gamma(t))^{\sharp}$, the conditions for the phase function and amplitudes follow from Lemmas \ref{lem:a1-AC22} and \ref{lem:a2-AC22}. The proof is done.
\end{proof}

\begin{proof}[Proof of Theorem \ref{thm:gbuc-AC22}]
    We will prove the theorem under the assumption that $(x,v) \in \tilde{\mathcal{G}}_T$. The case of self-intersections will be handled below in Section \ref{subsec:UCse-AC22}. If we denote the function in Theorem \ref{thm:gbu-AC22} by $w$, we take $u = c_n w$ where $c_n$ is the constant in Theorem \ref{thm:gbu-AC22}. It is then enough to prove the estimate \eqref{eq:conc-AC22} for $u$.
    
    By Theorem \ref{thm:gbu-AC22} we know that $u$ is of the form
    \[
    u = c_n h^{-\frac {n-1} 4} e^{\I \Phi/h} (a_0 + h a_1 + \ldots + h^N a_N) \rho.
    \]
    We will work in Fermi coordinates $(t,y)$ in $U$, where $U$ is the set 
    \begin{equation}
        U = \{ (t, y) \,:\, t \in (-\delta,\tau(x,v)+\delta), \ y \perp \dot{\gamma}_{x,v}(t), \ |y| < \delta \}.
    \end{equation}
    Since $\rho$ is supported in $U$, we can represent $u$ in the coordinates $(t,y)$ by
    \[
    u(t,y) = c_n h^{-\frac {n-1} 4} e^{\I \Phi(t,y)/h} (a_0(t,y) + h a_1(t,y) + \ldots + h^N a_N(t,y)) \rho(t,y).
    \]
    Also, $u \equiv 0$ outside $U$. We denote $ v = c_n h^{-\frac {n-1} 4} e^{\I \Phi(t,y)/h} a_0(t,y) \rho
(t,y)$ and write \[ u=v+hw, \quad \mbox{where} \quad w= c_n h^{-\frac {n-1} 4} e^{\I \Phi(t,y)/h} \sum_{k=1}^N h^{k-1}a_k(t,y). \] Then $ |v|^2= c_n^2  h^{-\frac {n-1} 2} e^{- 2\mathrm{Im}(\Phi(t,y))/h} |a_0|^2(t,y) \rho^2$.  In the coordinates $(t,y)$  we have the following expression for $ \Phi$:
\[\Phi(t,y)= t +\frac{1}{2} H(t)y\cdot y+ \mathcal{O} (|y|^3)  \implies \mathrm{Im}(\Phi) = \mathrm{Im}(\frac{1}{2} H(t)y\cdot y+ \mathcal{O} (|y|^3) ).\]
Next, we consider
\begin{align}
    \int_M |u|^2 \varphi \dif{V_g}
    & = \int_0^{\tau} \int_{\R^{n-1}}\varphi(t,y) |u|^2 |g|^{1/2} \df t\, \df y \nonumber \\
    &= \int_0^{\tau} \int_{\R^{n-1}} \varphi(t,y) \,|v|^2 |g|^{1/2} \dif t \dif y \nonumber \\
    & \quad + \int_0^{\tau} \int_{\R^{n-1}} \varphi(t,y) \,[|hw|^2+ 2 h\mathrm{Re}(v\bar{w})] |g|^{1/2} \dif t \dif y \nonumber \\
    & = J_1 + J_2. \label{eq_acc_6.8}
\end{align}
We analyze $J_1$ and $J_2$ separately. We start with $J_1$, which may be written as 
\begin{align*}
    & J_1
    = \int_0^{\tau} \int_{\R^{n-1}} \varphi(t,y) \,|v|^2 |g|^{1/2} \df t\, \df y\\&= c_n^2 h^{-\frac{n-1}{2}} \int_0^{\tau} \int_{\R^{n-1}} \varphi(t,y) \,  e^{ -2 \mathrm{Im}( \Phi (t,y))/h}\, |a_0|^2\, \rho^2(t,y)\,|g|^{1/2} \df t\, \df y \\
    & = c_n^2 h^{-\frac{n-1}{2}} \int_0^{\tau} \int_{\R^{n-1}} \varphi(t,y) \,  e^{ -\mathrm{Im}(H(t))y\cdot y/h} \, e^{ \mathcal{O}(|y|^3)/h}\, |a_0|^2\, \rho^2(t,y)\,|g|^{1/2} \df t\, \df y\\
    & = c_n^2 
    \int_0^{\tau} \int_{\R^{n-1}} e^{ -\mathrm{Im}(H(t))y\cdot y}\, \varphi(t,\sqrt{h} y) \,  e^{\sqrt{h} \mathcal{O}(|y|^3)}\, |a_0|^2(t,\sqrt{h} y)\, \rho^2(t,\sqrt{h} y)\, |g|^{\frac 1 2}(t, \sqrt{h} y)  \dif t \dif y. 
\end{align*}
Next we denote $ \Tilde{\varphi}(t,\sqrt{h}y):=   \varphi(t,\sqrt{h} y) \,  e^{\sqrt{h} \mathcal{O}(|y|^3)}\, |a_0|^2(t,\sqrt{h} y)\, \rho^2(t,\sqrt{h} y)\, |g|^{\frac 1 2}(t, \sqrt{h} y)$. By applying Taylor's theorem we obtain
\begin{equation*}
    \Tilde{\varphi}(t,\sqrt{h}y) = \Tilde{\varphi}(t,0)+ \sqrt{h} y \cdot \nabla_y \Tilde{\varphi}(t,\alpha y\sqrt{h}) \quad \mbox{for certain $\alpha \in (0,1)$}. 
\end{equation*}
Note that $ \Tilde{\varphi} (t,0)= \varphi(t,0) |a_0|^2(t,0)$ using the fact that $ |g|(t,0)=1$, and $ \rho(t,0)=1$. This implies
\begin{align*}
    J_1
    &= c_n^2 \int_0^{\tau} \int_{\R^{n-1}} e^{ -\mathrm{Im}(H(t))y\cdot y}\, \big[ \varphi(t,0) |a_0|^2(t,0) + \sqrt{h} y \cdot \nabla_y \Tilde{\varphi}(t,\alpha y\sqrt{h}) \big] \dif y \dif t \\
    & = c_n^2 \Big( \int_{\R^{n-1}} e^{-|y|^2} \int_0^{\tau} \frac {\varphi(t,0) |a_0|^2(t,0)}{\sqrt{\det(\mathrm{Im}(H(t)))}}+ \int_0^{\tau} \int_{\R^{n-1}} e^{ -\mathrm{Im}(H(t))y\cdot y}\,\sqrt{h} y \cdot \nabla_y \Tilde{\varphi}(t,\alpha y\sqrt{h}) \Big).
\end{align*}
Recall from Theorem \ref{thm:gbu-AC22} that  $c_n = (\int_{\R^{n-1}} e^{-\abs{y}^2} \dif y)^{-1/2}$. This entails
\begin{equation*}
    J_1
    = \int_0^{\tau} \varphi(t,0) \frac{|a_0|^2(t,0)}{\sqrt{\det(\mathrm{Im}(H(t)))}}+ c_n^2\int_0^{\tau} \int_{\R^{n-1}} e^{ -\mathrm{Im}(H(t))y\cdot y}\,\sqrt{h} y \cdot \nabla_y \Tilde{\varphi}(t,\alpha y\sqrt{h}).
\end{equation*}
As in \cite{Dos_Jems}*{p.\ 2599} we have that $\frac{|a_0|^2(t,0)}{\sqrt{\det(\mathrm{Im}(H(t)))}} $ is constant, and by our choices of initial data we have $ \frac{|a_0|^2(t,0)}{\sqrt{\det(\mathrm{Im}(H(t)))}}=1 $. From the above we deduce
\begin{equation} \label{eq_acc_6.9}
    J_1 = I\varphi(x,v) + c_n^2 \int_0^{\tau} \int_{\R^{n-1}} e^{-\mathrm{Im}(H(t)) y \cdot y}\, \sqrt{h} y \cdot \nabla_y \Tilde{\varphi}(t, \alpha y \sqrt{h}).
\end{equation}
Combining \eqref{eq_acc_6.8} and \eqref{eq_acc_6.9} we obtain that
\begin{equation} \label{eq_acc_6.10}
    \int_M |u|^2 \varphi \dif{V_g} - I\varphi(x,v)
    = c_n^2 \int_0^{\tau} \int_{\R^{n-1}} e^{ -\mathrm{Im}(H(t)) y \cdot y}\, \sqrt{h} y \cdot \nabla_y \Tilde{\varphi}(t,\alpha y\sqrt{h}) + J_2.   
\end{equation}
By utilizing the norm estimates from Lemmas \ref{lem:a1-AC22} and \ref{lem:a2-AC22} we conclude that 
\begin{align*}
   \int_0^{\tau} \int_{\R^{n-1}} \abs{e^{ -\mathrm{Im}(H(t))y\cdot y}\,\sqrt{h} y \cdot \nabla_y \Tilde{\varphi}(t,\alpha y\sqrt{h}) }\le C \sqrt{h} \nrm[C^1(M)]{\varphi}.
\end{align*}
Since $ \nrm[L^2(M)]{ h^{-\frac {n-1} 4} e^{\I \Phi(t,y)/h}}\le C$ and $ \nrm[C^N(M)]{a_j}\le C$, we have
\begin{equation*}
    J_2 \le C \sqrt{h} \nrm[C^1(M)]{\varphi}.
\end{equation*}
Thus, by combining the preceding estimate with \eqref{eq_acc_6.10}, we conclude 
\begin{align*}
    \abs{\int_M |u|^2 \varphi \dif{V_g} - I\varphi(x,v) }\le C \sqrt{h} \nrm[C^1(M)]{\varphi}.
\end{align*}
This completes the proof.
\end{proof}

Finally, we give the proofs of Lemmas \ref{lem:a1-AC22} and \ref{lem:a2-AC22}. The conditions for $\Phi$ and $a_r$ in these lemmas will be equivalent with the fact that $\nabla^j \Phi$ and $\nabla^j a_r$ solve certain ODEs along $\gamma_{x,v}$. To derive these ODEs we will use properties of the covariant derivative $\nabla_X$ and total covariant derivative $\nabla$ (see e.g.\ \cite{M_Lee_book}). These include the formulas $\nabla T(X, \,\cdot\,) = (\nabla_X T)(\,\cdot\,)$ and 
\begin{multline*}
(\nabla_X T)(X_1, \ldots, X_k) = X(T(X_1, \ldots, X_k)) \\
 - T(\nabla_X X_1, X_2, \ldots, X_k) - \ldots - T(X_1, \ldots, X_{k-1}, \nabla_X X_k).
\end{multline*}
We will also use that $\nabla_X$ commutes with contractions and raising and lowering of indices with respect to $g$. If $S$ is a $p$-tensor and $T$ is a $q$-tensor, we will use the special contraction 
\[
C(S,T)= c_{p,p+1}(S \otimes T^{\sharp})
\]
where $T^{\sharp}$ is obtained from $T$ by raising the first index and $c_{p,p+1}$ contracts the $p$th and $(p+1)$th indices. Equivalently 
\[
C(S,T)(X_1, \ldots, X_{p-1}, Y_1, \ldots, Y_{q-1}) = \sum_{j=1}^N S(X_1, \ldots, X_{p-1},E_j) T(E_j,  Y_1, \ldots, Y_{q-1})
\]
where $\{E_j\}$ is any orthonormal basis. Below we will also write $R(X,Y)T = (\nabla_X \nabla_Y - \nabla_Y \nabla_X - \nabla_{[X,Y]})T$ and $R_V(X,Y) = \langle R(X,V)V, Y \rangle$.

The following general Riemannian geometry identities will give the invariant ODEs for $\nabla^k \Phi$.

\begin{lemma} \label{lem_phi_id}
Let $\Phi$ be a smooth complex function on $M$, and let $G = \mathrm{grad}(\Phi) = (\df \Phi)^{\sharp}$. Then 
\[
\nabla_G (\nabla^2 \Phi) + C(\nabla^2 \Phi, \nabla^2 \Phi) + R_G = \frac{1}{2} \nabla^2(\langle G, G \rangle).
\]
If $H = (\nabla^2 \Phi)^{\sharp}$ is the $(1,1)$-tensor corresponding to $\nabla^2 \Phi$, this identity can be rewritten as 
\[
\nabla_G H + H^2 + R_G^{\sharp} = \frac{1}{2} (\nabla^2(\langle G, G \rangle))^{\sharp}.
\]
For any $k \geq 3$ one has 
\[
\nabla_G(\nabla^k \Phi) + A_k(\nabla^k \Phi) + F_k = \frac{1}{2} \nabla^k(\langle G, G \rangle).
\]
Here $A_k$ is a linear map taking $k$-tensors to $k$-tensors with $|A_k(S)| \leq C_k |\nabla^2 \Phi| |S|$. Moreover, $F_k$ is a $k$-tensor with $|F_k| \leq D_k$ where $D_k$ only depends on curvature quantities on $(M,g)$ and on $|\nabla^j \Phi|$ for $1 \leq j \leq k-1$.
\end{lemma}
\begin{proof} 
Since $G = (\df \Phi)^{\sharp}$ is a gradient field, we have for any $X$, $Y$ that 
\begin{equation} \label{gfs}
\nabla^2 \Phi(X,Y) = \langle \nabla_X G, Y \rangle = \langle \nabla_Y G, X \rangle.
\end{equation}
We compute 
\begin{align*}
    &\nabla^2(\langle G, G \rangle)(X,Y) = \nabla_X (\nabla(\langle G, G \rangle))(Y) = X(Y(\langle G, G \rangle)) - (\nabla_X Y)(\langle G, G \rangle) \\
    &= 2 X(\langle \nabla_Y G, G \rangle) - 2 \langle \nabla_{\nabla_X Y} G, G \rangle \\
    &= 2 \langle \nabla_X \nabla_Y G, G \rangle + 2 \langle \nabla_Y G, \nabla_X G \rangle - 2 \langle \nabla_G G, \nabla_X Y \rangle.
\end{align*}
On the other hand, we have 
\begin{align*}
    &\nabla_G(\nabla^2 \Phi)(X,Y) = G(\langle \nabla_X G, Y \rangle) - \langle \nabla_Y G, \nabla_G X \rangle - \langle \nabla_X G, \nabla_G Y \rangle \\
    &= \langle \nabla_G \nabla_X G, Y \rangle - \langle \nabla_Y G, \nabla_G X \rangle \\
    &= \langle \nabla_X \nabla_G G, Y \rangle + \langle \nabla_{[G,X]} G, Y \rangle + \langle R(G,X)G, Y \rangle - \langle \nabla_Y G, \nabla_G X \rangle
\end{align*}
where we used the definition of the curvature tensor $R(G,X)G$. To simplify the last expression, we apply $X$ to the identity $\langle \nabla_G G, Y \rangle = \langle \nabla_Y G, G \rangle$ obtained from \eqref{gfs} to see that 
\begin{align*}
    \langle \nabla_X \nabla_G G, Y \rangle &= \langle \nabla_X \nabla_Y G, G \rangle + \langle \nabla_Y G, \nabla_X G \rangle - \langle \nabla_G G, \nabla_X Y \rangle \\
    &= \frac{1}{2} \nabla^2(\langle G, G \rangle)(X,Y).
\end{align*}
Thus we obtain, using \eqref{gfs} and the fact that $[G,X] = \nabla_G X - \nabla_X G$,  
\begin{align*}
    &\nabla_G(\nabla^2 \Phi)(X,Y) \\
    &= \frac{1}{2} \nabla^2(\langle G, G \rangle)(X,Y) + \langle \nabla_{Y} G, [G,X] \rangle + \langle R(G,X)G, Y \rangle - \langle \nabla_Y G, \nabla_G X \rangle \\
    &=  \frac{1}{2} \nabla^2(\langle G, G \rangle)(X,Y) - \langle \nabla_Y G, \nabla_X G \rangle - R_G(X,Y).
 \end{align*}
Since $C(\nabla^2 \Phi, \nabla^2 \Phi) = \sum \langle \nabla_X G, E_j \rangle \langle \nabla_Y G, E_j \rangle = \langle \nabla_X G, \nabla_Y G \rangle$, this proves the identity for $\nabla^2 \Phi$.

We next apply $\nabla$ to the identity for $\nabla^2 \Phi$. We also use the identity  
\begin{align*}
    \nabla(\nabla_G T)(X,\,\cdot\,)
    & = \nabla_X \nabla_G T(\,\cdot\,)
    = \nabla_G \nabla_X T + \nabla_{[X,G]}T + R(X,G)T \\
    & = \nabla_G(\nabla T)(X,\,\cdot\,) + \nabla T(\nabla_G X, \,\cdot\,) + \nabla_{[X,G]}T + R(X,G)T \\
    & = \nabla_G(\nabla T)(X,\,\cdot\,)  + \nabla T(\nabla_X G, \,\cdot\,) + R(X,G)T \\
    & = \nabla_G(\nabla T)(X,\,\cdot\,)  + C(\nabla^2 \Phi, \nabla T)(X, \,\cdot\,) + R(X,G)T
\end{align*}
as well as 
\begin{align*}
    \nabla(C(S,T))(X,\,\cdot\,)
    & = \nabla_X(c_{p,p+1}(S \otimes T^{\sharp}) = c_{p,p+1}(\nabla_X S \otimes T^{\sharp} + S \otimes (\nabla_X T)^{\sharp}) \\
    & = C(\nabla S, T)(X,\,\cdot\,) + \sigma C(S, \nabla T)(X,\,\cdot\,)
\end{align*}
where $\sigma$ is a permutation, so that $\sigma S(X_1, \ldots, X_p) = S(X_{\sigma(1)}, \ldots, X_{\sigma(n)})$. Thus we obtain 
\[
\nabla_G(\nabla^3 \Phi) + C(\nabla^2 \Phi, \nabla^3 \Phi) + C(\nabla^3 \Phi, \nabla^2 \Phi) + \sigma C(\nabla^2 \Phi, \nabla^3 \Phi) + F_3 = \frac{1}{2}\nabla^3(\langle G, G \rangle)
\]
where $F_3$ contains terms depending on curvature quantities and on $\nabla^j \Phi$ for $1 \leq j \leq 2$. This is the required equation for $\nabla^3 \Phi$. The cases $k \geq 4$ proceed in an analogous way.
\end{proof}

We can now prove the main lemma for the phase function $\Phi$.

\begin{proof}[Proof of Lemma \ref{lem:a1-AC22}]
    The first requirements for $\Phi$ are the conditions $\Phi(\gamma(t)) = t$ and  $\df \Phi(\gamma(t)) = \dot{\gamma}(t)^{\sharp}$. In order to prescribe higher derivatives for $\Phi$ along $\gamma$ it is convenient to work with tensors along $\gamma$ that only act in directions orthogonal to $\dot{\gamma}$. For any $r, s \geq 0$ we define a smooth vector bundle $E = E^{r,s}$ over $\gamma$ such that the fiber $E_{\gamma(t)}^{r,s}$ is the space of multilinear forms on $(\dot{\gamma}(t)^{\perp})^{\otimes r} \otimes ((\dot{\gamma}(t)^{\perp})^*)^{\otimes s}$. Note that any tensor $A$ at $\gamma(t)$ gives rise to an element $A|_{\dot{\gamma}^{\perp}}$ of $E_{\gamma(t)}$, and conversely any element of $E_{\gamma(t)}$ can be identified with the corresponding tensor at $\gamma(t)$ that vanishes in the $\dot{\gamma}(t)$ direction. Using this identification one can compute $\nabla_{\dot{\gamma}}$ of a section of $E$, and one can check from the definitions that this produces another section of $E$ (this uses $\nabla_{\dot{\gamma}} \dot{\gamma} = 0)$. Similarly, one can raise and lower indices of sections of $E$ with respect to $g$. Below we will assume these conventions and work with tensors only acting in the $\dot{\gamma}^{\perp}$ directions.
    
    Next we require that $\nabla^2 \Phi|_{\dot{\gamma}(t)^{\perp}} = H(t)^{\flat}$, where $H(t)$ solves on $[0,\tau(x,v)]$ the ODE 
    \begin{equation} \label{h_eq_proof}
    \nabla_{\dot{\gamma}} H + H^2 + R_{\dot{\gamma}}^{\sharp} = 0, \qquad H(0) = \I (g|_{\dot{\gamma}(0)^{\perp}})^{\sharp}.
    \end{equation}
    We will also require that $G_k(t) := \nabla ^k\Phi|_{\dot{\gamma}(t)^{\perp}}$ for $3 \leq k \leq N$ solves the ODE 
    \begin{equation} \label{phik_eq_proof}
\nabla_{\dot{\gamma}}G_k + A_k(G_k) + F_k = 0, \qquad G_k(0) = 0.
    \end{equation}
    As discussed below, \eqref{h_eq_proof} and \eqref{phik_eq_proof} have unique solutions. By Lemma \ref{lem_phi_id}, the function $\Phi$ will then satisfy the conditions in Lemma \ref{lem:a1-AC22} except perhaps the uniformity of constants. Thus it remains to verify that the constants are uniform. The main part of the proof will be to verify that the arguments in \cite{Lassas2001boundary}*{Lemma 2.56} for solving the matrix Riccati equation are also valid in our case when the equation is written invariantly.
    
    To this end, let $Z(t)$ and $Y(t)$ be $(1,1)$-tensors along $\gamma = \gamma_{x,v}$ acting on $\dot{\gamma}^{\perp}$ that satisfy the following linear system of ODEs for $t \in [0,\tau(x,v)]$:
    \begin{alignat*}{2}
        \nabla_{\dot{\gamma}} Y & = Z, & Y(0) & = I, \\
        \nabla_{\dot{\gamma}} Z & = - R_{\dot \gamma}^{\sharp}Y, \quad & Z(0) & = \I (g|_{\dot{\gamma}(0)^{\perp}})^{\sharp}.
    \end{alignat*}
    This is a linear system and $\abs{R_{\dot \gamma}^{\sharp}} \leq C$, where $C$ denotes a constant that is uniform over $(x,v) \in \mathcal{G}_T$ and $t \in [0,\tau(x,v)]$ and may change from line to line. By energy estimates \cite{Taylor_book}*{Section 1.5} and by the fact that $\tau(x,v) \leq T$, it follows that $\abs{Y} + \abs{Z} \leq C$ uniformly.

    We wish to prove the uniform bound 
    \begin{equation} \label{yt_unif_bound}
    |Y(t)w|
    \geq C^{-1} |w|, \qquad w \perp \dot{\gamma}(t).
    \end{equation}
    To this end we first note the following Leibniz rule: if $A(t)$ and $B(t)$ are $(1,1)$-tensors and $c$ contracts the second and third indices, then 
    \[
    \nabla_{\dot{\gamma}} (AB) = \nabla_{\dot{\gamma}} (c(A \otimes B)) = c(\nabla_{\dot{\gamma}} A \otimes B) + c(A \otimes \nabla_{\dot{\gamma}} B) = (\nabla_{\dot{\gamma}} A)B + A(\nabla_{\dot{\gamma}} B).
    \]
     Then the argument in \cite {Lassas2001boundary}*{Lemma 2.57}, together with the fact that $R_{\dot \gamma}^{\sharp}$ is real and symmetric, gives that 
    \[
    \nabla_{\dot{\gamma}}(Z^t Y - Y^t Z) = \nabla_{\dot{\gamma}}(Z^* Y - Y^* Z) = 0.
    \]
    Here $Y^t$ and $Y^*$ etc are defined as $\langle Y v,w \rangle = \langle v, Y^t w \rangle$ and $(Y v, w) = (v, Y^* w)$, where $(v,w)$ is the sesquilinear $g$-inner product on complex tangent vectors. In particular, if $v(t)$ is a complex vector that is parallel along $\gamma$, this implies that  
    \begin{align*}
    2 \I \, \mathrm{Im} \big( Y(t) v(t), Z(t) v(t) \big) & = \big( (Z(t)^* Y(t) - Y(t)^* Z(t)) v(t), v(t) \big) \\
    & = \big( (Z(0)^* Y(0) - Y(0)^* Z(0)) v(0), v(0) \big) \\
    &= -2 \I g(v(0), \overline{v(0)}).
    \end{align*}
    If $v(t) = w$ where $w \perp \dot{\gamma}(t)$, then also $v(0) \perp \dot{\gamma}(0)$, and since $|v(s)|^2$ is constant in $s$ we have 
    \begin{equation} \label{im_ytw_ztw}
    \mathrm{Im} \big( Y(t) w, Z(t) w \big) = - |w|^2
    \end{equation}
    whenever $t \in [0,\tau(x,v)]$ and $w \perp \dot{\gamma}(t)$. In particular,
    \[
    |w|^2 \leq |Y(t)w|\,|Z(t)w| \leq C |Y(t)w|\,|w|
    \]
    using the uniform bound $|Z(t)| \leq C$. This proves \eqref{yt_unif_bound}.


    Now we can define $H(t)$ by 
    \[
    H(t) v = Z(t)Y(t)^{-1} v, \qquad v \perp \dot{\gamma}(t),
    \]
where $Y(t)^{-1}$ denotes the inverse of $Y(t)$ on $\dot{\gamma}(t)^{\perp}$ which exists by \eqref{yt_unif_bound}.  It follows that $H(t)$ solves \eqref{h_eq_proof} and satisfies $|H(t)| \leq C$ uniformly over $t \in [0,\tau(x,v)]$. Moreover, for $w \perp \dot{\gamma}(t)$ one obtains from \eqref{im_ytw_ztw} and the estimate $|Y(t)| \leq C$ that  
\[
(\mathrm{Im}(H(t))w, w) = \mathrm{Im}(Z(t)Y(t)^{-1}w,w) = |Y(t)^{-1}w|^2 \geq C^{-2} |w|^2.
\]
Thus $\nabla^2 \Phi$ satisfies the uniform estimate in \eqref{eq:aq2-AC22}. The linear ODEs \eqref{phik_eq_proof} are uniquely solvable with uniform bounds by energy estimates \cite{Taylor_book}*{Section 1.5}. This concludes the proof of the lemma.
\end{proof}

The proof of Lemma \ref{lem:a2-AC22} concerning the amplitudes proceeds in a similar way.
But prior to the proof of Lemma \ref{lem:a2-AC22}, we present the following result first.

\begin{lemma}\label{lem:amplitiude}
    For any smooth function $a$, one has  
    \begin{align*}
        \nabla(La)= L (\nabla a) + B( \nabla a) + F a
    \end{align*}
    where $B$ is a smooth linear map satisfying $|B(\nabla a)| \leq C |\nabla a|$ and $F$ satisfies $|F| \leq C$, with $C$ only depending on $(M,g)$ and $\Phi$. For any  integer $k\ge 2$, one has 
    \begin{align*}
        \nabla^k (La) = L(\nabla^k a) + B_k(\nabla^ka) +F_k,
    \end{align*}
    where $B_k$ is a linear map  satisfying $ \abs{B_k(\nabla^k a)}\le C \abs{\nabla^k a}$. Moreover, $|F_k| \leq C$ where $C$ depends on $\Phi$, curvature quantities on $(M,g)$ and $ |\nabla^j a|$ for $ 1\le j\le k-1$.
\end{lemma}
\begin{proof}
We first extend the notion of $L$ for any vector field or tensor field as 
\[
LZ= \frac 1 \I ( 2\nabla_GZ + \Delta_g \Phi Z).
\]
We next compute $\nabla(La) - L (\nabla a)$. To this end, we consider
\(
\nabla (La)(X)
=\nabla_X(La)
= -\I (2\nabla_X \nabla_G a + \nabla_X (\Delta_g \Phi a))
= -\I [2\nabla_G \nabla_X a + 2\nabla_{[X,G]} a+ (X\Delta_g \Phi) a + (\Delta_g \Phi) (Xa) ].
\)
Since $ [X,G]= \nabla_XG-\nabla_G X $, this along with preceding equation entails
\begin{align} \label{eq_a.2}
    \nabla (La)(X)
    = -\I [ 2\nabla_G\nabla_X a +2 \nabla a (\nabla_X G)-2\nabla a(\nabla_G X) + (X\Delta_g \Phi) a + \Delta_g \Phi (Xa) ].    \end{align}
We also have 
\begin{align}\label{eq_a.3}
    L(\nabla a)(X)
    &  = -\I (2\nabla_G\nabla a + \Delta_g \Phi \nabla a)(X)\notag\\
    &= -\I [2\nabla_G\nabla_X a - 2\nabla a (\nabla_G X)+ (\Delta_g \Phi) \nabla a (X)]. 
\end{align}
The combination of \eqref{eq_a.2} and \eqref{eq_a.3} entails
\begin{align*}
    \nabla(La)(X)= L (\nabla a) (X)+ \frac{2}{\I}\nabla a(\nabla_X G) +\frac{1}{\I} (\nabla_X \Delta_g \Phi) a. 
\end{align*}
To prove the relation for higher order derivatives, we will employ induction argument and assume that \begin{align}\label{induction}
    \nabla^k(La) = L(\nabla^k a) + B_k(\nabla^ka) +F_k,
\end{align}
holds for any smooth function  $a$, where $|B_k|, |\nabla B_k|, |F_k|, |\nabla F_k| \leq C$. Then for $ k+1$ applying $\nabla$ to the equation \eqref{induction} we obtain 
\begin{align*}
\nabla(\nabla^kLa) (X) = \nabla_XL(\nabla^k a)+ \nabla_X B_k(\nabla^k a)   +\nabla_XF_k. 
\end{align*}
Then it remains to compute formula for $ \nabla_XL(\nabla^k a)$ and $ L(\nabla^{k+1} a)(X,\cdots)$. This can be done in a similar fashion as above. This completes the induction argument.
\end{proof}

\begin{proof}[Proof of Lemma \ref{lem:a2-AC22}]
As in the proof of Lemma \ref{lem:a1-AC22}, we work with tensors that only act in the $\dot{\gamma}^{\perp}$ directions. By the construction of the  amplitude $a=a_0+ha_1+\ldots+ h^k a_k$ we have that $ \nabla^j (La_0)=0$ on $\gamma$ for $ 0\le j\le N$. This along with Lemma \ref{lem:amplitiude} entails that $ \nabla^j a_0$ satisfies the linear ODE
\begin{align}\label{eq_transport}
 L(\nabla^j a_0|_{\dot{\gamma}^{\perp}})+ B_j(\nabla^j a_0|_{\dot{\gamma}^{\perp}}) +F_j=0   \end{align} 
along the geodesic  $\gamma$ for $ 0\le j\le N$. By energy estimates \cite{Taylor_book}*{Section $1.5$} and by the fact that $ \tau(x,v) \le T$, we conclude that $ \abs{\nabla^j a_0|_{\dot{\gamma}^{\perp}}} \le C$ uniformly. This further entails $ \nrm[C^N(M)]{a_0} \le C$ uniformly. Moreover, we have that $\nabla^j ( La_r- \Delta a_{r-1})=0 $ for $1\le j\le N$ and $1\le r\le N$. Utilizing this and Lemma \ref{lem:amplitiude}, one can obtain certain linear ODEs for $ \nabla^j a_r|_{\dot{\gamma}^{\perp}}$ (similar to \eqref{eq_transport} with different source terms). Hence by standard energy estimates we conclude that $ \nrm[C^N(M)]{a_r}\le C$ uniformly for $ 1\le r\le  N$. 

 It remains to prove \eqref{principle_term}. By Lemma \ref{lem:a2-AC22} (a)
 we have that $La_0|_{\gamma(t)}=0$. Recall that $ Lv := \frac{1}{\I}(2 \langle \df \Phi, \df v \rangle + (\Delta_g \Phi)v).  $  This along with $La_0|_{\gamma(t)}=0$  implies  \[\nabla_{\dot{\gamma}} a_0(\gamma(t)) + \frac{1}{2} \Delta_g\Phi(\gamma(t))\, a_0(\gamma(t))=0 \quad \mbox{and} \quad a(\gamma(0))=1. \]  
This is a linear ODE along $\gamma$ with an  initial  condition. It has a unique solution and it is given by 
\begin{equation*}
    a_0(\gamma(t))
    = \exp \Big[ -\frac{1}{2} \int_0^t \Delta_g\Phi(s)\dif s \Big]
    = \exp \Big[ -\frac{1}{2} \int_0^t \mathrm{tr}_g(H(s)) \dif s \Big].
\end{equation*}
In the last part we used $\nabla^2 \Phi(\dot{\gamma}, w) = \langle \nabla_{\dot{\gamma}} \nabla \Phi, w \rangle = 0$ since $\nabla \Phi = \dot{\gamma}$. This completes the proof.
\end{proof}

\subsection{Self-intersection case} \label{subsec:UCse-AC22}

We now describe an extension procedure that allows us to reduce the proofs of Theorems \ref{thm:gbuc-AC22} and  \ref{thm:gbu-AC22} in the general case to the case where the geodesics do not self-intersect, so that the self-intersection case can be handled.

Recall that $(M,g)$ is a domain with boundary in the closed manifold $(S,g)$, which has positive injectivity radius $\mathrm{inj}(S)$ \cite{M_Lee_book}*{Lemma 6.16}. Below we write $\mathrm{inj}(M) = \mathrm{inj}(S)$.
We first give an upper bound for the number of self-intersection points for geodesics in $M$ with length $\leq T$.

\begin{lemma} \label{lem:upsT-AC22}
    Let $(M,g)$ be a compact oriented Riemannian manifold with smooth boundary and let $T > 0$. 
    There is an uniform upper bound on the number of self-intersection points for all geodesics $\gamma_{x,v}$ with $(x,v) \in \mathcal{G}_T$.
\end{lemma}

\begin{proof}
    Let $\gamma \colon [0,\tau(x,v)] \to M$  be a geodesic.
    Since $\tau(x,v)$ is the length of the geodesic $\gamma$, we can divide $\gamma$ into
    \begin{equation} \label{eq:L2tu-AC22}
    L_{x,v} = \frac{2\tau(x,v)}{\mathrm{inj}(M)}
    \end{equation}
    pieces such that each piece is of length $\mathrm{inj}(M)/2$ except perhaps for the last piece, and we denote each piece as $\gamma_1, \ldots, \gamma_{L_{x,v}}$ sequentially.
    Note that each $\gamma_i$ is not self-intersecting.
    Also, for $j \neq i$, $\gamma_i$ and $\gamma_j$ intersect only once.
    To see this, if $\gamma_i$ and $\gamma_j$ intersect at two different points $x_0$ and $y_0$, then there are two distinct geodesics connecting these two points.
    This contradicts the definition of $\mathrm{inj}(M)$.
    Therefore, there are at most $L_{x,v}-1$ self-intersections happening for each $\gamma_i$. Since $\tau(x,v) \leq T$ for $(x,v) \in \mathcal{G}_T$, we define 
    \[
    L = \frac{2T}{\mathrm{inj}(M)}
    \]
    and observe that the number of intersection points for all the geodesics with $(x,v) \in \mathcal{G}_T$ is bounded by $L(L-1)$.
    The proof is done.
\end{proof}

Our next lemma gives a uniform lower bound on the angles between segments of a geodesic at a self-intersection point. 

\begin{lemma}\label{lem:uniform angle}
    Given $T>0$ there exists $\epsilon>0$ such that for any $ (x,v)\in \mathcal{G}_T$ with $\gamma_{x,v}(t)= \gamma_{x,v}(s),\, s\neq t$ one has $\langle \dot{\gamma}_{x,v}(s),\dot{\gamma}_{x,v}(t)\rangle\le 1- \epsilon$.   
\end{lemma}
\begin{proof}
    We argue by contradiction. Let us assume that for all $j$, there exists $(x_j,v_j) \in \mathcal{G}_T$ and $(s_j,t_j) \in [0,\tau(x_j,v_j)]$ with $s_j < t_j$ such that 
 \begin{align}\label{eq_6.22}
     \gamma_{x_j,v_j}(s_j)  = \gamma_{x_j,v_j}(t_j),\quad \langle \dot{\gamma}_{x_j,v_j}(s_j),\dot{\gamma}_{x_j,v_j}(t_j)\rangle\ge 1-\frac{1}{j}. 
 \end{align}
Since $ \PD_{+}SM$ is compact, there exists a subsequence, denoted also as $(x_j,v_j)$, that converges to $(x,v)\in \PD_{+}SM.$ Since $ \tau$ is upper semi-continuous and $\tau(x,v)\le T$, this further implies that $(x,v)\in \mathcal{G}_T$.
By further choosing a subsequence, we can assume $s_j$ and $t_j$ converge to $s$ and $t$ respectively with $s \leq t$, and since $ s_j,t_j \in [0,\tau(x_j,v_j)]$, we know $s,t \in [0,\tau(x,v)]$. Using these, by taking the limit $j \to \infty$ in \eqref{eq_6.22},
 we conclude that
 \begin{align*}
     \gamma_{x,v}(s)=\gamma_{x,v}(t), \quad \langle\dot{\gamma}_{x,v}(s),\dot{\gamma}_{x,v}(t)\rangle=1 \implies  \gamma_{x,v}(s)=\gamma_{x,v}(t),\quad\dot{\gamma}_{x,v}(s)= \dot{\gamma}_{x,v}(t).
 \end{align*}
Now we have two possibilities: either $s < t$ or $s = t$. If $s < t$, then $\gamma_{x,v}$ makes a loop, but this contradicts with the fact that $\tau(x,v)\le T$. If $s=t$, then $s_j$ and $t_j$ get simultaneously   close to $s$, this is again a contradiction because $\mbox{inj}(M) >0$.
We complete the proof.
\end{proof}

\smallskip

We are going to glue many copies of subsets of $M$ together so that the geodesic $\gamma$ does not intersect itself in the new manifold.
Then we can apply the results in Section \ref{subsec:UCno-AC22} to obtain uniform bounds for self-intersecting geodesics.
We refer to \cite{StefanovUhlmann2008}*{Section 2.1} for similar ideas.

Let $N = L(L-1)$ and let $t_j = \frac j N \tau(x,v)$.
Let $\{ (U_j, \phi_j) \}_{j=0}^{N+1}$ be an open cover of $\gamma$, where $U_j = \phi_j^{-1}(I_j \times B)$ and $I_j = (t_j - 2\epsilon, t_{j+1} - \epsilon)$ with $j = 0,\cdots,N$.
In this way, we see each $\{ (U_j, \phi_j) \}$ is a chart of $\gamma_j$.
We write as $(\tilde U_j, \tilde \phi_j)$ copies of charts $(U_j, \phi_j)$, and we also copy the Riemann tensor structure from  $(U_j, \phi_j)$ to $(\tilde U_j, \tilde \phi_j)$, say, we copy $g$ as $\tilde g$.
We want to glue different $\tilde U_j$ together.
To that end, let us investigate the intersection parts $U_j \cap U_{j+1}$ for different $j$.
We denote
\begin{equation} \label{eq:VjWj-AC22}
    V_j := \tilde \phi_j^{-1} \circ \phi_j (U_j \cap U_{j+1}) \subset \tilde U_j
    \quad \textrm{and} \quad
    W_j := \tilde \phi_{j+1}^{-1} \circ \phi_{j+1} (U_j \cap U_{j+1}) \subset \tilde U_{j+1}.
\end{equation}
Then $V_j$ is the copy of $U_j \cap U_{j+1}$ in $\tilde U_j$, and $W_j$ is the copy of $U_j \cap U_{j+1}$ in $\tilde U_{j+1}$.
There is a diffeomorphism between $V_j$ and $W_j$ because they are both copies of $U_j \cap U_{j+1}$.
According to \eqref{eq:VjWj-AC22}, this diffeomorphism can be expressed as
\begin{equation} \label{eq:idVW-AC22}
    W_j = \tilde \phi_{j+1}^{-1} \circ \phi_{j+1} \circ \phi_j^{-1} \circ \tilde \phi_j (V_j).
\end{equation}
Now we identify $V_j$ and $W_j$ though this diffeomorphism.
This identification allows us to construct the following manifold $\tilde M$:
\[
\tilde M := \bigsqcup_{j=0}^{N+1} \tilde U_j
\]
Note that the boundary of $\tilde M$ is not smooth yet.
This can be seen illustratively from the points $A$ and $B$ in Fig.~\ref{fig:int-AC22}.
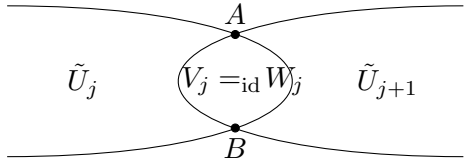
\begin{figure}[ht]
    \begin{tikzpicture}
    \draw (-3,-1) .. controls (2,-1) and (2,1) .. (-3,1);
    \draw (3,-1) .. controls (-2,-1) and (-2,1) .. (3,1);
    
    \node at (-2,0) {$\tilde U_j$};
    \node at (2,0) {$\tilde U_{j+1}$};
    \node at (0.1,0) {$V_j =_{\mathrm{id}}\! W_j$};
    
    \filldraw (0,0.62) circle (1.5pt) node[anchor=south] {$A$};
    \filldraw (0,-0.62) circle (1.5pt) node[anchor=north] {$B$};
    \end{tikzpicture}
    \caption{Intersection between $\tilde U_j$ and $\tilde U_{j+1}.$ Here the notation $=_{\mathrm{id}}$ means equal by identification through \eqref{eq:idVW-AC22}.}
    \label{fig:int-AC22}
\end{figure}

\begin{lemma} \label{lem:Msmo-AC22}
    $(\tilde M,\tilde g)$ is an $n$-dimensional compact Riemannian manifold with boundary.
\end{lemma}

\begin{proof}
    Because there are overlaps between $U_j$ and $U_{j+1}$ for $j = 0,\cdots,N$,
    and because $\tilde U_j$ are copies of $U_j$, we see that there is a natural identification between the neighborhoods of the boundary $\tilde U_j \cap \tilde U_{j+1}$ in $\tilde U_j$ and in $\tilde U_{j+1}$, respectively.
    This guarantees that $\tilde M$ is a compact Riemannian manifold with dimension the same as in $\tilde M$.
\end{proof}

We shall show that $\tilde M$ can be trimmed to become a manifold with smooth boundary.
To this end, we first introduce the geodesic $\tilde \gamma$ and prove it does not intersect itself.
Denote $\tilde \gamma$ to be the curve in $\tilde M$ with coordinates $\bigcup_{j=0}^{N} I_j \times \{0\}$, that is to say,
\[
\tilde \gamma
\colon \bigcup_{j=0}^{N} I_j \ \to \ \bigcup_{j=0}^{N} \tilde \phi_j^{-1}(I_j \times \{0\}),
\quad
\tilde \gamma
\colon t \in I_j \ \mapsto \ \tilde \phi_j^{-1}(t \times \{0\}).
\]
We call $\tilde \gamma$ a \emph{lifting} of $\gamma$.

\begin{lemma} \label{lem:tgg-AC22}
    $\tilde \gamma$ is a geodesic in $\tilde M$.
\end{lemma}

\begin{proof}
    For a point $\tilde p \in \tilde \gamma$ that belongs to the non-gluing part of certain $\tilde U_j$, we can identity a small neighborhood of $\tilde p$ with that of $p \in \gamma$ in $M$.
    Therefore, we only need to investigate points $\tilde p_j$ at $\tilde \gamma \cap \tilde U_j \cap \tilde U_{j+1}$ for $j = 0,\cdots, N$.
    The point $\tilde p_j$ belongs to $\tilde U_j$, and thus its neighborhood can also be identified to a neighborhood certain point in $\gamma$.
    Therefore, $\tilde \gamma$ satisfies the geodesic equation in $\tilde M$ just as $\gamma$ does in $M$, and so $\tilde \gamma$ is a geodesic in $\tilde M$.
\end{proof}

\begin{lemma} \label{lem:tit-AC22}
    $\tilde \gamma$ does not intersect itself in $\tilde M$.
\end{lemma}

\begin{proof}
    We argue by contradiction.
    Assume $\tilde \gamma(t_1) = \tilde \gamma(t_2)$ for $t_1 \neq t_2$.
    If $t_1$, $t_2$ belong to the same interval $I_j$, then we can conclude $\tilde \phi_j^{-1}(t_1 \times \{0\}) = \tilde \phi_j^{-1}(t_2 \times \{0\})$, which gives $t_1 = t_2$ because $\phi_j^{-1}$ is a diffeomorphism.
    But this contradicts with $t_1 \neq t_2$.
    If $t_1 \in I_j$ and $t_2 \in I_k$ for $j \neq k$, then $\tilde \gamma(t_1) = \tilde \phi_j^{-1}(t_j \times \{0\}) \subset \tilde U_j$ and $\tilde \gamma(t_2) = \tilde \phi_k^{-1}(t_k \times \{0\}) \subset \tilde U_k$.
    From \eqref{eq:L2tu-AC22} we see it is impossible that $|j-k| = 1$, because the injectivity radius covers two consecutive pieces, thus $|j-k| \geq 2$.
    Therefore, $\tilde U_j$ and $\tilde U_k$ are two distinct sets who do not share any gluing part, so it is impossible that $\tilde \gamma(t_1) = \tilde \gamma(t_2)$.
    In both cases, we have a contradiction.
    The proof is done.
\end{proof}

\begin{lemma} \label{lem:Msm-AC22}
    $\tilde M$ contains a subset $\bar M$ such that $(\bar M,\tilde g)$ is a $n$-dimensional compact Riemannian manifold with smooth boundary.
\end{lemma}

\begin{proof}
    The first part of the proof follows from Lemma \ref{lem:Msmo-AC22},
    because $\tilde \gamma$ does not intersect itself in $\tilde M$, we can construct a global coordinates $\tilde \psi \colon M \to \Rn$ according to $\tilde \gamma$.
    Let $\kappa \in (0,1)$ be a small enough constant such that $\bar M := \psi^{-1}(I \times B(0,\kappa))$ is a subset of $\tilde M$.
    This $\bar M$ has a smooth boundary.
    See Fig.~\ref{fig:fil-AC22} as an illustration.
    The proof is done.
\end{proof}

\begin{figure}[ht]
    \centering
    \begin{tikzpicture}
    \draw (-3,-1) .. controls (2,-1) and (2,1) .. (-3,1);
    \draw (3,-1) .. controls (-2,-1) and (-2,1) .. (3,1);
    
    \draw (-3,0) -- (3,0) node[anchor=west] {$\gamma_{x,v}$};
    \draw (-3,0.4) -- (3,0.4) node[anchor=west] {$a$};
    \draw (-3,-0.4) -- (3,-0.4) node[anchor=west] {$b$};
    
    \fill [opacity=0.3, pattern=north west lines] (-3,-0.4) rectangle (3,0.4);
    
    \node at (-3.3,0.8) {$\tilde U_j$};
    \node at (3.5,0.8) {$\tilde U_{j+1}$};
    
    \filldraw (0,0.6) circle (1.5pt) node[anchor=south] {$A$};
    \filldraw (0,-0.6) circle (1.5pt) node[anchor=north] {$B$};
    \end{tikzpicture}
    \caption{Let $\bar M$ be the area in between lines $a$ and $b$.
    Compared to $\tilde M$ whose boundary may not be smooth at points $A$ and $B$, $\bar M$ has a smooth boundary.}
    \label{fig:fil-AC22}
\end{figure}
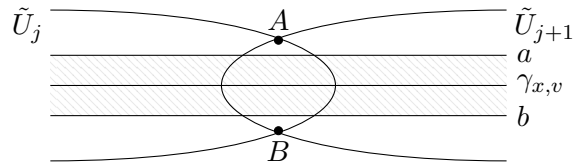
With a slight abuse of notation, we refer $\tilde M$ as this $\bar M$ with smooth boundary.
This shall bring no trouble to the following analysis, as we only consider the situation in small enough neighborhoods of $\tilde \gamma$ in $\tilde M$, and of $\gamma$ in $M$.
We call $(\tilde M, \tilde g)$ a \emph{lifting manifold of $(\tilde M, \tilde g)$ with respect to $\gamma$}.
We define a map $\Psi$ from $\tilde M$ to $\bigcup_{j=0}^{N+1} U_j \subset M$ by
\begin{equation} \label{eq:Pijj-AC22}
    \Psi \colon \tilde M \to \bigcup_{j=0}^{N+1} U_j,
    \quad
    \Psi \colon p \in \tilde U_j \ \mapsto \ \phi_j^{-1} \circ \tilde \phi_j(p).
\end{equation}
It can be checked that $\Psi(\tilde \gamma) = \gamma$.

\begin{lemma} \label{lem:tgni-AC22}
    $\forall p \in \tilde M$, the set $\Psi^{-1}(p)$ is finite with cardinality  $\le N+2$.
    Moreover, $\Psi$ is a local diffeomorphism.
\end{lemma}

\begin{proof}
    Assume that $q$, $q' \in \tilde M$ are distinct points with $p = \Psi(q) = \Psi(q')$.
    According to the definition of $\tilde M$, we have $q \in \tilde U_j$ and $q' \in \tilde U_k$ for certain $0 \leq j,k \leq N+1$.
    This gives $\phi_j(p) = \tilde \phi_j(q)$ and $\phi_k(p) = \tilde \phi_k(q')$.
    If $j = k$, then from the diffeomorphism property of $\tilde \phi_j$ we can conclude $q = q'$.
    This contradicts with our assumption that $q \neq q'$.
    This shows that elements in $\Psi^{-1}(p)$, if there are more than two, must come from different sets of $\tilde U_j~(j = 0,1,\cdots, N+1)$.
    Therefore, there are at most $N+2$ elements in the set $\Psi^{-1}(p)$.
    
    Locally $\Psi$ is defined by $\Psi(p) := \phi_j^{-1} \circ \tilde \phi_j(p)$ if $p \in \tilde U_j$.
    Both $\phi_j^{-1}$ and $\tilde \phi_j$ are locally diffeomorphic, thus $\Psi$ is also a local diffeomorphism.
    The proof is done.
\end{proof}

\smallskip

Now, we apply the arguments in Section \ref{subsec:UCno-AC22} to $\tilde \gamma$ and construct a quasimode $\tilde u := e^{\I \lambda \Phi} a$ which is more precisely given in the form \eqref{u_nsi_form}.
Denote
\[
u(p)
:= \sum_{s \in \Psi^{-1}(p)} \tilde u(s), \qquad
\forall p \in \bigcup_{j=0}^{N+1} U_j \subset M.
\]
The geodesic $\tilde \gamma$ does not intersect itself on $(\tilde M, \tilde g)$, thus by Proposition \ref{thm:gbu-AC22} and Lemmas \ref{lem:a1-AC22} and \ref{lem:a2-AC22} we can find uniform bounds $\tilde c$, $\tilde C$ which depend only on the geometric structure of $(\tilde M, \tilde g)$.
Because for each $\tilde p \in \tilde M$ we can find a neighborhood that is locally diffeomorphic to that of $M$, we see that the local geometric structure of $(\tilde M, \tilde g)$ is preserved with respect to $(M,g)$.
Because the construction for the phase function $\Phi$ and the amplitude $a$ is local,
thus for the self-intersecting geodesic $\gamma$ on $M$, we can also find the corresponding  bounds $c =  \tilde c \times\# (\Psi^{-1}(p))$ and $C =  \tilde C \times \# (\Psi^{-1}(p))$. These bounds are uniform over all $ (x,v) \in \mathcal{G}_T$  due to Lemmas \ref{lem:tgni-AC22} and \ref{lem:upsT-AC22}. Theorem \ref{thm:gbu-AC22} follows from this also in the case of self-intersecting geodesics.

\begin{proof}[Proof of Theorem \ref{thm:gbuc-AC22} for self-intersection cases]
    Let $\gamma$ be a geodesic in $(M,g)$ which may intersect itself.
    Let us construct a lifting manifold $(\tilde M, \tilde g)$ with respect to $\gamma$,
    and let $\tilde \gamma$ be the lifting of $\gamma$ in $(\tilde M, \tilde g)$.
    By Lemmas \ref{lem:tgg-AC22} and \ref{lem:tit-AC22} we know $\tilde \gamma$ is a non-self-intersecting geodesic in $(\tilde M, \tilde g)$.
    Let us also define $\Psi$ according to \eqref{eq:Pijj-AC22}.
    For any $\varphi \in C^1(M)$, we denote $\tilde \varphi := \varphi \circ \Psi$, thus $\tilde \varphi \in C^1(\tilde M)$.
    Therefore, according to Theorem \ref{thm:gbuc-AC22} for non-self-intersection cases, we have
    \begin{equation*}
        \left\lvert \int_{\tilde M} |\tilde u|^2 \tilde \varphi \dif V_{\tilde g} - I \tilde \varphi(x,v) \right\rvert \leq C \norm{\tilde \varphi}_{C^1(\tilde M)} h^{1/2}.
    \end{equation*}
    where $I \tilde \varphi$ represents the geodesic ray transform of $\tilde \varphi$ with respect to $\tilde \gamma$ in $\tilde M$.
    By definition we can have
    \begin{equation*}
    I \tilde \varphi(x,v)
    = \int_{0}^{\tau(x,v)} \varphi \circ \Psi(\tilde \gamma(t,x,v)) \dif t
    = \int_{0}^{\tau(x,v)} \varphi(\gamma(t,x,v)) \dif t
    = I \varphi(x,v),
    \end{equation*}
    where we have used $\Phi(\tilde \gamma) = \gamma$.
    
    From the definition $\tilde \varphi := \varphi \circ \Psi$ we see $\norm{\tilde \varphi}_{C^1(\tilde M)} \leq C \norm{\varphi}_{C^1(M)}$ for certain constant $C$.

    Furthermore, we know
    \(
    u(p)
    := \sum_{s \in \Psi^{-1}(p)} \tilde u(s),
    \)
    thus when $\Psi^{-1}(p)$ contains only one element, say $s$, we have $|u(p)|^2 = |\tilde u(s)|^2$;
    when $\Psi^{-1}(p)$ contains multiple elements, say $\Psi^{-1}(p) = \{s_1, \cdots, s_\ell\}$, then $|u(p)|^2 = \sum_{j=1}^\ell |\tilde u(s_j)|^2 + \sum_{j=1}^\ell \sum_{k \neq j}^\ell \tilde u(s_j) \overline {\tilde u(s_k)}$.
    By Lemma \ref{lem:tgni-AC22} we know $\ell \leq N+2$.
    Since all the intersections are transversal and by Lemma \ref{lem:uniform angle} there is a uniform lower bound on the angles of intersection, one can apply a non-stationary phase argument as in \cite{Dos_Jems}{equation (3.6)} to obtain that the integrals of $\tilde u(s_j) \overline {\tilde u(s_k)}$ for $j \neq k$ are of $\mathcal O(h)$.  
    Combining these arguments, we arrive at
    \begin{equation*}
        \left\lvert \int_M |u|^2 \varphi \dif V_g - I \varphi(x,v) \right\rvert \leq C \norm{\varphi}_{C^1(M)} h^{1/2}.
    \end{equation*}
    The proof is done.
\end{proof}


{


}

\end{document}